\documentclass[3p,twocolumn,final]{elsarticle}

\usepackage{graphicx}
\usepackage{amssymb}
\usepackage{amscd}
\usepackage{amsmath}
\usepackage{amsthm}
\usepackage[noend]{algpseudocode}
\usepackage{algorithm}
\usepackage{enumerate}
\usepackage[all]{xy}
\usepackage[inline]{trackchanges}
\usepackage{comment}
\usepackage{subfigure}
\usepackage{url}

\theoremstyle{plain}  
\newtheorem{thm}{Theorem}[section]
\newtheorem{lem}[thm]{Lemma}
\newtheorem{prop}[thm]{Proposition}
\newtheorem{cor}[thm]{Corollary}

\theoremstyle{definition}  
\newtheorem{defn}[thm]{Definition}

\newtheorem{ex}[thm]{Example}

\theoremstyle{remark}  
\newtheorem{rem}[thm]{Remark}

\newcommand{\R}{{\mathbb{R}}}

\newcommand{\Z}{{\mathbb{Z}}}

\newcommand{\cX}{{\mathcal X}}

\newcommand{\phom}{H}
\newcommand{\pdiag}{D}

\newcommand{\dbn}{d_{\rm b}}
\newcommand{\dgh}{d_{\rm GH}}
\newcommand{\dhaus}{d_{\rm H}}
\newcommand{\dinfty}{d_{\infty}}

\newcommand{\rank}{\hbox{\rm rank}}
\newcommand{\image}{\mathop{\rm im}}
\newcommand{\kernel}{\mathop{\rm ker}}

\newcommand{\realpositive}{\mathbb{R}_{\geq 0}}

\definecolor{gray85}{gray}{0.85} 
\definecolor{gray8}{gray}{0.8} 
\definecolor{gray7}{gray}{0.7} 
\definecolor{gray6}{gray}{0.6} 
\definecolor{gray5}{gray}{0.5} 
\definecolor{gray4}{gray}{0.4} 
\definecolor{gray35}{gray}{0.35} 

\usepackage{MnSymbol}

\numberwithin{equation}{section}
\def\pivot{\textrm{pivot}}

\newcommand{\delaunay}{{\rm Del}}
\newcommand{\vr}{{\rm VR}}
\newcommand{\alp}{{\rm Alp}}

\begin{document}

\begin{frontmatter}

\title{Continuation of Point Clouds via Persistence Diagrams}

\author[usp]{Marcio Gameiro} 
\ead{gameiro@icmc.usp.br}

\author[aimr]{Yasuaki Hiraoka}
\ead{hiraoka@wpi-aimr.tohoku.ac.jp}

\author[aimr]{Ippei Obayashi}
\ead{ippei.obayashi.d8@tohoku.ac.jp}

\address[usp]{Instituto de Ci\^{e}ncias Matem\'{a}ticas e de Computa\c{c}\~{a}o,
Universidade de S\~{a}o Paulo, Caixa Postal 668, 13560-970, S\~{a}o Carlos, SP, Brazil}

\address[aimr]{WPI-Advanced Institute for Materials Research (WPI-AIMR), Tohoku
University, 2-1-1 Katahira, Aoba-ku, Sendai, 980-8577 Japan}

\begin{abstract}
In this paper, we present a mathematical and algorithmic framework for the continuation of point clouds
by persistence diagrams. A key property used in the method is that the persistence map, which assigns
a persistence diagram to a point cloud, is differentiable. This allows us to apply the Newton-Raphson
continuation method in this setting. Given an original point cloud $P$, its persistence diagram $\pdiag$,
and a target persistence diagram $\pdiag'$, we gradually move from  $\pdiag$ to $\pdiag'$, by successively
computing intermediate point clouds until we finally find a point cloud $P'$ having $\pdiag'$ as its persistence
diagram. Our method can be applied to a wide variety of situations in topological data analysis where it is
necessary to solve an inverse problem, from persistence diagrams to point cloud data.
\end{abstract}



\end{frontmatter}

\section{Introduction}
\label{sec:introduction}

Let $P$ be a finite set of points in $\R^L$ given by
\begin{equation}
\label{eq:pointclouds}
P=\{u_i\in \R^L\mid i=1\dots,M\}.
\end{equation}
We call $P$ a point cloud, following the convention in topological data analysis (TDA) \cite{carlsson,eh}. TDA
provides us tools to study the ``shape" of $P$. Among them, persistent homology \cite{elz,zc} is one of the most
useful tools, and it is now applied into various practical applications, e.g., amorphous solids \cite{glass_letter,iop},
proteins \cite{protein}, and sensor networks \cite{ghrist} (see also \cite{carlsson} and references therein).

Persistent homology can be regarded as a collection of maps, called \emph{persistence maps} in this paper,
from $P$ to a finite set  $D_\ell$, for $\ell=0,1,\cdots$, in the extended plane $\bar{\R}^2=\bar{\R}\times\bar{\R}$,
where $\bar{\R}=\R\cup\{\infty\}$. The set $D_\ell$ is called \emph{persistence diagram} and it encodes the
$\ell$-dimensional topological features of $P$ with metric information (precise definitions are given in
Section~\ref{sec:ph}).

In many applications, the point cloud $P$ have an intricate ``shape'' or structure, and, in this situations,
persistence diagrams are used to provide the ``essential'' topological features of $P$. For example, in the
papers \cite{glass_letter,iop}, the authors study hierarchical geometric structures in several amorphous solids.
In such a case, $P$ is given by an atomic configuration of an amorphous solid and consists of thousands of
points in $\R^3$ obtained by molecular dynamics simulations. It is a difficult task to directly study the geometry
and physical properties of the amorphous solid from $P$ due to its immense size. Hence, a key observation
of their work is that the persistence diagrams of the atomic configurations can capture essential geometric
information of the amorphous solids. From this significant property, using persistence diagrams they obtain
various physical properties of the solid, such as, finding new hierarchical ring structures, decompositions of
first sharp diffraction peaks, mechanical responses, etc.

Figure~\ref{fig:glass_pd} shows a schematic representation of $D_1$ for silica glass, $P$, studied in
\cite{glass_letter} (this corresponds to Figure~1 in that paper). They show that the presence of curves in
$D_1$ precisely distinguishes the amorphous state from liquid and crystalline states. It means that the normal
directions to these curves express characteristic geometric constraints on atomic configurations of amorphous
states. Therefore, changing $D_1 \Rightarrow D'_1$ along a normal direction (e.g., black to red in
Figure~\ref{fig:glass_pd}) and tracing the corresponding deformation $P \Rightarrow P'$ in the atomic
configurations clarify the geometric origin of rigidity in the amorphous solid, which is currently an important
problem in physics and material sciences. For this purpose, we need to solve an inverse problem: given $D'_1$
find $P'$ in some appropriate setting, and this is the main subject of this paper.
\begin{figure}[ht]
\centering
\includegraphics[width=0.6\hsize]{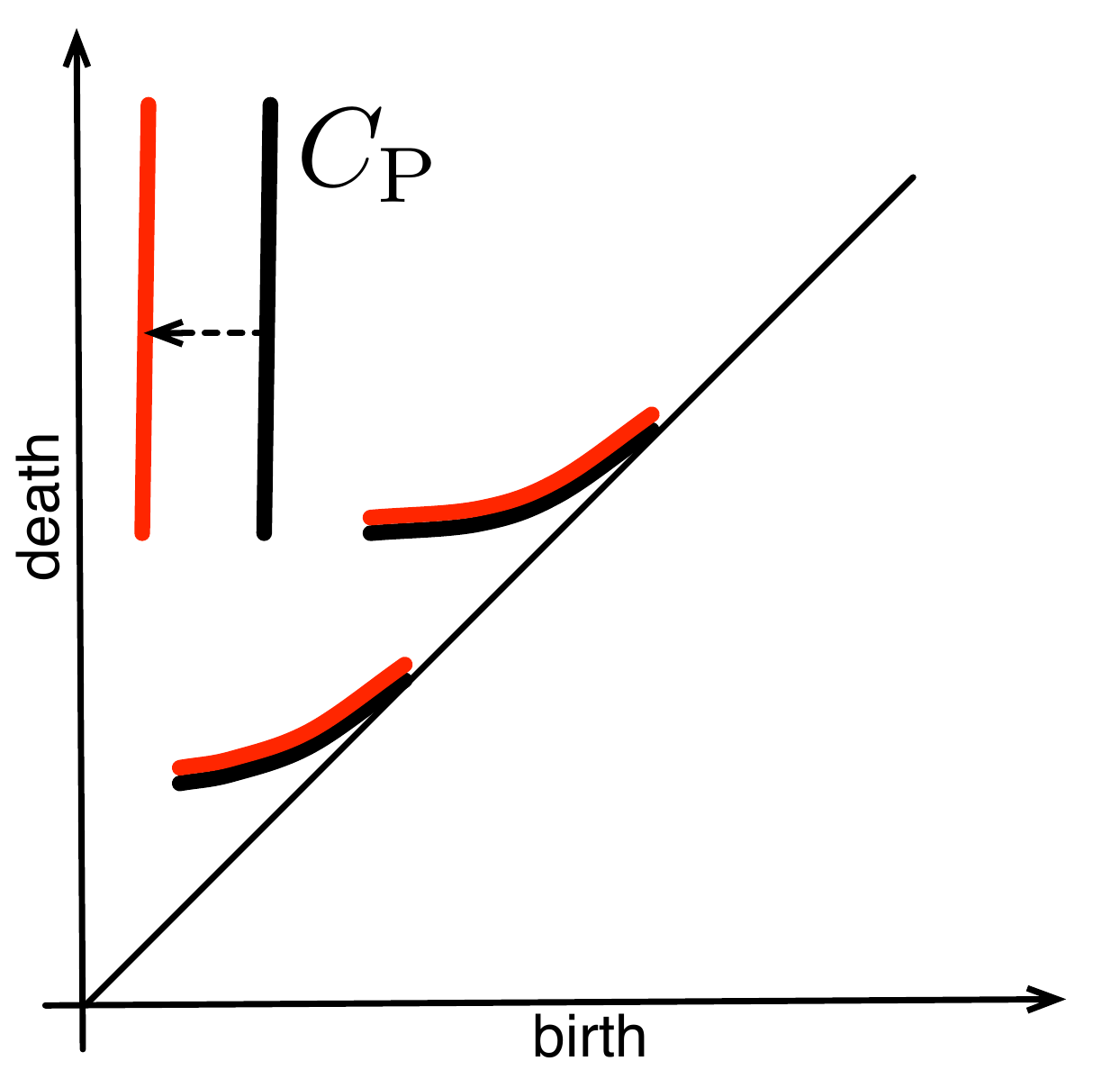}
\caption{Schematic representation of persistence diagrams. \textbf{Black:} $D_1$ for silica glass
(see \cite{glass_letter} for details). \textbf{Red:} A target persistence diagram $D_1'$ to study the
geometric constraints generating the curve $C_{\rm P}$. }
\label{fig:glass_pd}
\end{figure}

In this paper, we present a mathematical and algorithmic framework for solving inverse problems of persistence
diagrams. Our method is based on the continuation method \cite{continuation1,continuation}, which was originally
developed in numerical bifurcation theory of dynamical systems, applied to the setting of persistent homology.
More precisely, given a known correspondence between a point cloud $P$ and its persistence diagram
$D_\ell$, and a target persistence diagram $D_\ell'$, we develop a method to obtain a point cloud $P'$ which have
$D_\ell'$ as its persistence diagram. We first represent the persistence diagrams as points in some Euclidean space
and divide the line segment $\overline{D_\ell D_\ell'}$ into small segments
$D_\ell^{(0)}=D_\ell, D_\ell^{(1)}, \dots, D_\ell^{(N-1)}, D_\ell^{(N)}=D_\ell'$. Then, we solve an implicit equation defined
by the persistence diagram $D_\ell^{(i)}$ for each small segment, using the Newton-Raphson method, to obtain a new
point cloud $P^{(i)}$ having $D_\ell^{(i)}$ as its persistence diagram. By successively applying this procedure, we finally
obtain a desired point cloud $P' = P^{(N)}$.

We remark that the inverse problem from a persistence diagram $D_\ell$ to a point cloud $P$ is not well-posed in general.
Namely, it is possible to have multiple point clouds giving the same $D_\ell$ (non-uniqueness). Furthermore, the target
persistence diagram may not be located in the image of the persistence map (non-existence). Our approach to these
issues is as follows: Regarding non-uniqueness, at each step of the continuation we try to find the point cloud closest
to the point cloud on the previous step of the continuation. This assigns a minimality condition on the Euclidean norm of
the difference for persistence diagrams, and provides the uniqueness property. Furthermore, this minimality condition is
reasonable in the practical applications mentioned above, since our input atomic configuration is usually realized as
a minimum point of a certain energy landscape, and hence, finding the closest atomic configuration to the minimum point
is a natural choice.  Regarding non-existence, we take a practical approach. Namely, we try to apply our continuation
method and, if the computation is successful we conclude that our target persistence diagram is in the image of the
persistence map. If not, we investigate the reason for non-convergence of the Newton-Raphson method, which could
be due to non-existence, the presence of zero singular values, etc (see Section~\ref{sec:zsv}). We note that it is a very
challenging mathematical problem to study the image of the persistence map. 

This paper is organized as follows. The fundamental concepts, such as simplicial complex models used to represent
the point clouds and persistent homology, are introduced in Section~\ref{sec:settings}. In Section~\ref{sec:lppm} local
properties of persistence maps, especially differentiability, are studied in detail. Section~\ref{sec:cont} is the core of the
paper and is devoted to developing the continuation method of point clouds using persistence diagrams. In
Section~\ref{sec:computations}, we show some computational examples of the proposed method. Finally, in
Section~\ref{sec:conclusion}, we conclude with a list of future improvements to our continuation method.

\section{Simplicial Complex Model and Persistent Homology}
\label{sec:settings}

\subsection{Simplicial Complex Models}

Let $V$ be a finite set, and $\Sigma$ be a collection of subsets of $V$. A \emph{simplicial complex}
is defined by a pair $(V,\Sigma)$ satisfying (i) $\{v\}\in\Sigma$ for all $v\in V$ and
(ii) $\sigma\in\Sigma$ and $\tau\subset\sigma$ imply $\tau\in\Sigma$. An element
$\sigma\in\Sigma$ with $|\sigma|=\ell+1$ is called an \emph{$\ell$-simplex}.

Let $P$ be a point cloud (\ref{eq:pointclouds}) in $\R^L$. 
For $0<r<\infty$, we refer to the open ball with radius $r$ as an $r$-ball,
and denote it, with center $u_i$, by
\[
B_r(u_i)=\{x\in\R^L\mid ||x-u_i||< r\}.
\]

The \emph{Vietoris-Rips complex} $\vr(P,r)$ of $P$ with radius $r$ is defined as the simplicial
complex $(P,\Sigma)$ where the set $\Sigma$ of simplices is determined by
\[
\sigma \in \Sigma \Longleftrightarrow B_r(u_s)\cap B_r(u_t)\neq\emptyset,~~\forall u_s, u_t\in\sigma.
\]

The definition of the Vietoris-Rips complex depends only 
on the distances of all pairs in $P$. Hence the Vietoris-Rips complex 
has an advantage that it is computable even if $L$ is large.

The \emph{alpha complex} $\alp(P, r)$ \cite{eh,em} is another simplicial complex model of $P$ defined by using the
set of $r$-balls $B_r(u_i)$. A significant property of the alpha complex is the homotopy equivalence
\[
\bigcup_{i=1}^M {B_{r}(u_{i})} \simeq |\alp(P,r)|,
\]
where $|\alp(P,r)|$ is a geometric realization of $\alp(P,r)$. Because of this property, the alpha complex is widely
used in practical applications to analyze topological features in $P$. We note that the Vietoris-Rips complex does
not satisfy this property in general. 

Fast software for computing alpha complexes in dimensions $L=2,3$ is available, e.g., \cite{cgal}. 
In this paper, we use alpha complexes for $L=3$, while the case $L=2$ can be similarly treated.

For $0<\alpha <\infty$, an $\alpha$-ball $B$ is called $P$-empty if $B \cap P= \emptyset$. For $\ell=0,1,2,3$,
let $\delaunay_\ell(P)$ be the set of $\ell$-simplices $\sigma \subset P$  such that there exists an $P$-empty open
ball $B$ with $\partial B \cap P = \sigma$.  The \emph{Delaunay triangulation} $\delaunay(P)$ of $P$
is the simplicial complex whose simplices are given by $\delaunay_\ell(P)$ for $\ell = 0,1,2,3$.

The three dimensional alpha complex is defined as a subcomplex of the Delaunay triangulation $\delaunay(P)$.
For each $\sigma \in \delaunay_\ell(P)$, let $B_\sigma$ be the smallest open ball with $\sigma \subset \partial B_\sigma$,
and $\rho_\sigma$ be the radius of $B_\sigma$.  Let us define $G_{0,\alpha} = P$, and $G_{\ell,\alpha}$ to be the
set of $\ell$-simplices $\sigma\in\delaunay_\ell(P)$ such that $B_\sigma$ is $P$-empty and $\rho_\sigma < \alpha$. A simplex in
$\bigcup_\alpha G_{\ell,\alpha}$ is called an attaching $\ell$-simplex. The \emph{alpha complex} $\alp(P,\alpha)$ is defined
as the simplicial complex whose simplices are given by $G_{\ell,\alpha}$ and their faces for $\ell=0,1,2,3$.
From this definition, the alpha complex $\alp(P,\alpha)$ is a subcomplex of the Delaunay triangulation $\delaunay(P)$,
and we have that $\delaunay(P)=\bigcup_\alpha\alp(P,\alpha)$.

We note that both simplicial complex models define \emph{filtrations of finite type}
\begin{align*}
  \vr(P)&=(\vr(P,r))_{r\in \realpositive}, \\
  \alp(P)&=(\alp(P,r))_{r\in\realpositive},
\end{align*}
where $\realpositive$ is the set of nonnegative reals. 
Namely, we have that $\vr(P,r)\subset \vr(P,s)$ and $\alp(P,r)\subset \alp(P,s)$ for $r<s$ and
$\vr(P,r)=\vr(P,R)$ and $\alp(P,r)=\alp(P,R)$ for $r\geq R$, for a sufficiently large $R$
(called \emph{saturation time}). The radius parameter $r$ is also called time in this paper,
following the convention used in persistent homology.

\subsection{General Position}
Let us treat a point cloud $P$ as an ordered set induced by the index $i=1,\dots,M$. 
Then, we can assign a single variable $u=(u_1,\dots,u_M) \in \R^n$ to $P$, where $n=LM$.
Conversely, from a point $u\in\R^n$, an ordered subset $P$ in $\R^L$ with $|P|=M$ can be constructed. 
We explicitly denote this correspondence by $u(P)$ and $P(u)$, if necessary, and identify them in the following. 

Let $\cX=(X^r)_{r\in \realpositive}$ be a Vietoris-Rips or an alpha filtration. For each simplex $\sigma\in X^R$, we
can assign its birth radius $r_\sigma$ in the filtration $\cX$ by the infimum radius $r$ satisfying $\sigma\in X^r$.

In the Vietoris-Rips filtration $\vr(P)$, the birth radius $r_\sigma$ of a simplex $\sigma=\{u_{i_0},\dots, u_{i_\ell}\}$
is a function of $u_\sigma=(u_{i_0},\dots,u_{i_\ell})$ given by
\[
r_\sigma=\frac{1}{2}\max_{0\leq a<b\leq \ell}||u_{i_b}-u_{i_a}||.
\]
We call an edge $\{u_{i_a},u_{i_b}\}$ that attains the above maximum an \emph{attaching edge} of $\sigma$.

\begin{defn}
\label{definition:vr_general}
A configuration $u \in \R^n$ is said to be in \emph{Vietoris-Rips general position} if the following conditions
are satisfied:
\begin{enumerate}
\item[(i)] $u_i \neq u_j$ for any $i\neq j$,
\item[(ii)] $r_\tau\neq r_{\tau'}$ for any attaching edges $\tau\neq\tau'$.
\end{enumerate}
The open set consisting of the points in Vietoris-Rips general position is denoted by $U_{\vr}$.
\end{defn}

In the alpha filtration $\alp(P)$, each simplex $\sigma$ appears as either an attaching simplex or a simplex
attached by some attaching simplex $\tau\supset \sigma$. In the latter case, the birth radius $r_\sigma$ is
given by $r_\sigma = r_\tau$.
\begin{defn}
\label{definition:alpha_general}
A configuration $u\in\R^n$ is said to be in \emph{alpha general position} if the following conditions are
satisfied:
\begin{enumerate}
\item[(i)] $u$ is in general position in the sense of \cite{em}.
\item[(ii)] $r_\tau\neq r_{\tau'}$ for any attaching simplices $\tau\neq\tau'$ except $0$-simplices.
\end{enumerate}
The open set consisting of the points in alpha general position is denoted by $U_{\alp}$.
\end{defn}

We note that, in both Vietoris-Rips and alpha filtrations, the condition (ii) implies that an attaching
simplex is uniquely determined by its birth radius.

\subsection{Persistent Homology}\label{sec:ph}
We briefly review the definition of persistent homology as a graded module on a monoid ring.
Let $\cX=(X^r)_{r\in\realpositive}$ be a filtration of finite type with a saturation time $R$. 
For each $X^r$, let us denote by $X^r_\ell$ the set
of $ \ell$-simplices in $X^r$. In the following, we fix an orientation for each simplex $\sigma=\{u_{i_0},\dots,u_{i_\ell}\}$ 
by $i_0<\cdots<i_\ell$,
and denote the oriented simplex by $\langle\sigma\rangle=\langle u_{i_0}\dots u_{i_\ell}\rangle$.

Let $k$ be a field, and let us treat $\realpositive$ with a monoid structure
induced by the addition $+$. Let $k[\realpositive]$ be a monoid ring. That is, $k[\realpositive]$ is a vector
space of formal linear combinations of elements of $\realpositive$ equipped with a ring structure
\[
(a_1r_1)\cdot (a_2r_2)=(a_1a_2)(r_1+r_2)
\]
for $a_1, a_2 \in k$ and $r_1, r_2 \in \realpositive$.

In the following, the elements in $k[\realpositive]$ are expressed by linear combinations of formal monomials $az^r$,
where $a\in k$, $r\in \realpositive$, and $z$ is an indeterminate. Then, the product of two elements are
defined by linear extension of 
\[
a_1z^{r_1}\cdot a_2z^{r_2}=a_1a_2z^{r_1+r_2}.
\]

Let us denote by $C_\ell(X^r)$ the $k$-vector space spanned by the $\ell$-simplices in $X^r_\ell$. 
The $\ell$-th chain group $C_\ell(\cX)$ is defined as a graded module on the monoid ring
$k[\realpositive]$ by taking a direct sum
\[
C_\ell(\cX)=\bigoplus_{r\in\realpositive}C_\ell (X^r)=
\{(c_r)\mid c_r\in C_\ell(X^r)\},
\]
where the action of a monomial $z^s$ on $C_\ell(\cX)$ is given by the right shift operator
\[
z^s\cdot (c_r)=(c'_r)~~{\rm with}~c'_r=\left\{
\begin{array}{cl}
c_{r-s},&r\geq s,\\
0,&r<s.
\end{array}
\right.
\]
For a simplex $\sigma$, let us define
\[
\llangle \sigma \rrangle = (c_r),~~~
c_r=\left\{\begin{array}{cl}
\langle\sigma\rangle,&r=r_{\sigma},\\
0,&r\neq r_{\sigma}.
\end{array}
\right. 
\]

We note that $\Xi_\ell=\{\llangle \sigma \rrangle \mid \sigma\in X^R_\ell\}$
forms a basis of $C_\ell(\cX)$. The boundary map $\partial_\ell: C_\ell(\cX)\rightarrow C_{\ell-1}(\cX)$ is
defined by linear extension of
\begin{equation*}
\label{eq:boundarymap}
\partial_\ell\llangle \sigma\rrangle=\sum_{j=0}^\ell(-1)^jz^{r_\sigma-r_{\sigma_j}}\llangle \sigma_j\rrangle,
\end{equation*}
where $\langle\sigma_j\rangle=\langle u_{i_0}\dots\widehat{u_{i_j}}\dots u_{i_\ell}\rangle$ is a face of
$\langle\sigma\rangle=\langle u_{i_0}\dots u_{i_\ell}\rangle$,
and $\widehat{a}$ means the removal of the vertex $a$.

The cycle group $Z_\ell(\cX)$ and the boundary group $B_\ell(\cX)$ in $C_\ell(\cX)$ are defined by
\[
Z_\ell(\cX)=\kernel \partial_\ell,~~~B_\ell(\cX)=\image\partial_{\ell+1}.
\]
It follows from $\partial_\ell\circ\partial_{\ell+1}=0$ that we have $B_\ell(\cX)\subset Z_\ell(\cX)$.
Then, the $\ell$-th \emph{persistent homology} is defined by
\[
\phom_\ell(\cX)=Z_\ell(\cX)/B_\ell(\cX).
\]

The following theorem is known as the structure theorem of persistent homology. 
\begin{thm}[\cite{zc}]
There uniquely exist indices $s,t\in\Z_{\geq 0}$ and $(b_i,d_i)\in\realpositive^2$, for $i=1,\dots,s$,
with $b_i<d_i$ and $b_i\in\realpositive$, for $i=s+1,\dots,s+t$, such that the following isomorphism
holds
\begin{equation}
\label{eq:decomposition}
\phom_\ell(\cX)\simeq\bigoplus_{i=1}^s
\left((z^{b_i})\biggl/(z^{d_i})
\right)\oplus\bigoplus_{i=s+1}^{s+t}(z^{b_i}),
\end{equation}
where $(z^a)$ expresses an ideal in $k[\realpositive]$ generated by the monomial $z^a$. 
When $s$ or $t$ is zero, the corresponding direct sum is ignored.
\end{thm}

The $\ell$-th \emph{persistence diagram} $\pdiag_\ell(\cX)$ of $\cX$ is
defined as the multiset in $\bar{\R}^2$ determined from the decomposition (\ref{eq:decomposition}) by
\begin{equation}
\label{eq:pd}
\pdiag_\ell(\cX)=\{(b_i,d_i)\mid i=1,\dots,s+t\},
\end{equation}
where $d_i=+\infty$ for $i=s+1,\dots,s+t$. The pair $(b_i, d_i)$ is called a \emph{birth-death pair} in
the $\ell$-th persistence diagram, and $b_i$, $d_i$ are called, respectively, the
\emph{birth} and \emph{death times} of the pair.

\section{Local Properties of the Persistence Map}
\label{sec:lppm}

\subsection{The Persistence Map}

Let $\pdiag_\ell(\cX_P)$ be the persistence diagram (\ref{eq:pd}) of a filtration $\cX_P$ constructed
from a finite set $P$. By choosing the birth and death times in (\ref{eq:pd}) that
are finite, we can express $\pdiag_\ell(\cX_P)$ as a point
\[
v=(b_1,d_1,\dots,b_s,d_s,b_{s+1},\dots,b_{s+t})\in\R^m,
\]
where $m=2s+t$. Then, recalling the identification of $P$ and $u \in \R^n$, we can regard the persistent
homology as giving a single correspondence
\begin{equation}
\label{eq:functionalform}
\R^n\ni u \longmapsto v \in \R^m.
\end{equation}
In this section, we define an appropriate open set $O \subset \R^n$ such that this single correspondence is
extended to a map $f \colon O \subset \R^n \rightarrow \R^m$ computing persistence diagrams with $f(u)=v$. 

It should be noted that the dimension $m$ may change for a different choice
of $u\in \R^n$. For extending the single correspondence to a map into $\R^m$, we use a recent result in \cite{cdo}.
Let us first recall some definitions.

For a metric space $(X,d_X)$, the \emph{Hausdorff distance} $\dhaus$ between two subsets $A,B\subset X$
is defined by 
\[
\dhaus(A,B) = \max\{\sup_{a\in A}d_X(a,B), \sup_{b\in B}d_X(b,A)\},
\]
where $d_X(a,B) =\inf_{b\in B}d_X(a,b)$ and $d_X(b,A)$ is defined symmetrically. The
\emph{Gromov-Hausdorff distance} $\dgh$ between two metric spaces $(X,d_X)$ and $(Y,d_Y)$ is defined by 
\[
\dgh(X,Y)=\inf_{f,g} \dhaus(f_{X\rightarrow Z}(X),g_{Y\rightarrow Z}(Y)),
\]
where $f_{X\rightarrow Z}$ and $g_{Y\rightarrow Z}$ denote isometric embeddings of $X$ and $Y$ into a
metric space $(Z,d_Z)$, respectively, and the Hausdorff distance between $f_{X\rightarrow Z}(X)$ and
$g_{Y\rightarrow Z}(Y)$ is measured using the metric $d_Z$.

We also recall that the \emph{bottleneck distance} $\dbn$ between two persistence diagrams $D$ and $D'$
is defined by
\[
\dbn(D,D')=\inf_\gamma\sup_{p\in \widehat{D}}||p-\gamma(p)||_\infty,
\]
where $\widehat{D}$ is a multiset consisting of the points in $D$ and the points on the diagonal $\Delta$
with multiplicity $+\infty$, and $\gamma$ is a bijection between $\widehat{D}$ and $\widehat{D'}$. 
Here, we define the norm $||p||_\infty=d_\infty(p,0)$ by the distance $d_\infty(p,q)=\max\{|p_1-q_1|,|p_2-q_2|\}$ on $\bar{\R}^2$.

For a multiset $D\subset\bar{\R}^2$, let 
\[
	T_\epsilon(D)=\{p\in D\mid \dinfty(p,\Delta)\geq \epsilon\}
\]
be the multiset defined by an $\epsilon$-truncation
of $D$ from the diagonal.

\begin{lem}
\label{lemma:nondecreasing}
Let $D$ be a persistence diagram and $\delta=d_\infty(D,\Delta)$. If
$\dbn(D,D')<\delta$, then $|D|\leq |D'|$. Furthermore, if $\dbn(D,D')<\epsilon<\delta/2$, then $|D|=|T_\epsilon(D')|$.
\end{lem}

\begin{proof}
Suppose $|D|> |D'|$. Then, there exists a point in $D$ which is mapped to the diagonal $\Delta$ for any
bijection $\gamma: \widehat{D}\rightarrow \widehat{D'}$. This leads to $\delta\leq \dbn(D,D')$, implying the first statement. 

For the proof of the second statement, let $\gamma$ be a bijection $\gamma:\widehat{D}\rightarrow \widehat{D'}$
such that
\[
\dbn(D,D')\leq \sup_{p\in\widehat{D}}||p-\gamma(p)||_\infty<\epsilon.
\]
Suppose $\dinfty(\gamma(D),\Delta)\leq \epsilon$. Then, there exists $p\in D$ such that $\dinfty(\gamma(p),\Delta)\leq\epsilon$.
On the other hand, we have $\dinfty(p,\gamma(p))\leq \sup_{p\in \widehat{D}}||p-\gamma(p)||_\infty<\epsilon$. This
leads to the contradiction 
\[
\delta\leq \dinfty(p,\Delta)\leq \dinfty(p,\gamma(p))+\dinfty(\gamma(p),\Delta)<2\epsilon<\delta.
\]
Hence, we have $\dinfty(\gamma(D),\Delta)> \epsilon$. Moreover, we have $T_\epsilon(D'\setminus \gamma(D))=\emptyset$,
otherwise it gives the contradiction $\sup_{p\in\widehat{D}}||p-\gamma(p)||_\infty\geq \epsilon$.
These two properties show that $|D|=|T_\epsilon(D')|$.
\end{proof}

Now let us apply the result in \cite{cdo}. Let $P$ and $P'$ be two point clouds in $\R^L$ with
$|P|=|P'|=M$. Then, since $P$ and $P'$ are totally bounded, the inequalities 
\begin{equation}
\label{eq:stability1}
\dbn(\pdiag_\ell(\cX_P),\pdiag_\ell(\cX_{P'}))\leq \dgh(P,P')
\end{equation}
for Vietoris-Rips filtrations and
\begin{equation}\label{eq:stability2}
\dbn(\pdiag_\ell(\cX_P),\pdiag_\ell(\cX_{P'}))\leq \dhaus(P,P')
\end{equation}
for alpha filtrations hold by Theorem 5.2 and Theorem 5.6 in \cite{cdo}, respectively.

Let $\delta=\dinfty(\pdiag_\ell(\cX_P),\Delta)$ and, for $\epsilon<\delta/2$, let us set
\[
O^\epsilon_u=\{u'\in\R^n\mid \dgh(P(u),P'(u'))< \epsilon\}
\]
for Vietoris-Rips filtrations, and
\[
O^\epsilon_u=\{u'\in\R^n\mid \dhaus(P(u),P'(u'))< \epsilon\}
\]
for alpha filtrations. Then, it follows from Lemma~\ref{lemma:nondecreasing} that
$|\pdiag_\ell(\cX_{P(u)})|= |T_\epsilon(\pdiag_\ell(\cX_{P'(u')}))|$ for any $u'\in O^\epsilon_u$.
Hence, given a single persistence correspondence $(u,v)$, we can define a map
\[
\Phi: O^\epsilon_u\ni u' \longmapsto T_\epsilon(\pdiag_\ell(\cX_{P'(u')}))\in\R^m.
\]

At the end of this section, we show that $O^\epsilon_u$ is an open set. We first prove the
following lemma.
\begin{lem}
\label{lemma:gh_euclid_inequality}
Let $P=\{x_1,\dots,x_M\}$ and $Q=\{y_1,\dots,y_M\}$ be two point clouds in $\R^L$. Set $u=u(P) \in \R^n$
and $w = w(Q) \in \R^n$. Then,
\[
\dgh(P,Q) \leq \dhaus(P,Q)\leq d_{\R^n}(u,w).
\]
\end{lem}

\begin{proof}
By definition of the Gromov-Hausdorff distance, we have
\begin{align*}
\dgh(P,Q)&\leq \dhaus(P,Q) \\
&=\max\{\max_id_{\R^L}(x_i,Q),\max_id_{\R^L}(P,y_i)\}.
\end{align*}
Without loss of generality, we assume that $\dhaus(P,Q)$ is achieved by $x_i$. Then,
\begin{align*}
\dhaus(P,Q)&=\min_jd_{\R^L}(x_i,y_j)\leq d_{\R^L}(x_i,y_i) \\ 
&\leq \sqrt{\sum_{i}d_{\R^L}(x_i,y_i)^2}= d_{\R^n}(u,w),
\end{align*}
and this completes the proof.
\end{proof}

\begin{prop}
$O^\epsilon_u$ is an open set in $\R^n$.
\end{prop}

\begin{proof}
We consider the case of Rips filtrations.
Let us choose $u'\in O^\epsilon_u$ and set $P'=P'(u') \subset \R^L$.
We also take a positive real number $\tau>0$ such that
\[
\dgh(P,P')+\tau<\epsilon.
\]
We claim that the $\tau$-open neighborhood of $u'$ in $\R^n$ is contained in $O^\epsilon_u$.
For any $u''\in\R^n$ with $d_{\R^n}(u',u'')<\tau$, we have $\dgh(P',P'')\leq d_{\R^n}(u',u'')<\tau$
from Lemma \ref{lemma:gh_euclid_inequality}. Then, it follows from the triangle inequality that
\begin{align*}
\dgh(P,P'')&\leq\dgh(P,P')+\dgh(P',P'')\\
&<\dgh(P,P')+\tau<\epsilon.
\end{align*}
This proves the claim and, since $u'$ is arbitrary, this concludes that $O^\epsilon_u$ is an open set in $\R^n$.
We can prove the case for alpha filtrations by replacing $\dgh$ by $\dhaus$.

\end{proof}

\subsection{Decomposition of Persistence Map}
Let $\cX=(X^r)_{r\in\realpositive}$ be a Vietoris-Rips or alpha filtration with a saturation time $R$, and let us set 
$W_\ell=X^R_\ell$ and $W=\bigcup_\ell W_\ell$.
The correspondence (\ref{eq:functionalform}) is decomposed into two parts:
\[
u \overset{g}{\longmapsto} r = (r_\sigma)_{\sigma\in W} \overset{h}{\longmapsto} v,
\]
where $g$ constructs the simplicial complex filtration, and $h$ computes the persistence diagram. We extend the decomposition
$f = h \circ g$ of this single correspondence to that of a map $f$. To this aim, we need to construct a proper
subset in $O^\epsilon_u$ in which the set $W$ is invariant. 

In the case of Vietoris-Rips filtration, $W$ is given by all the faces of the $(M-1)$-simplex for $|P|=M$,
independently of the configuration $u$. Thus, the decomposition $f=h\circ g$ of the correspondence is
extended naturally to the map
\begin{equation*}
\label{eq:map_decomp}
O_{\vr} \overset{g}{\longrightarrow} \R^{d} \overset{h}{\longrightarrow} \R^m,
\end{equation*}
where $O_{\vr}=O^\epsilon_u\cap U_{\vr}$, $g : O_{\vr}\ni u \longmapsto r=(r_\sigma)_{\sigma\in W} \in \R^{d}$, and
$h: \R^{d}\ni r \longmapsto v \in \R^m$ with $d=|W|$.

On the other hand, in the case of alpha filtration, note that the set $W$ can generally change depending
on the configuration $u$. Recall that the alpha complex $\alp(P(u),\alpha)$ is a subcomplex of the
Delaunay complex $\delaunay(P(u))$ and $\delaunay(P(u))=\bigcup_\alpha\alp(P(u),\alpha)$. Hence, the set $W$
is given by $\delaunay(P(u))$ in this case. 

For a configuration $u$ satisfying the general position assumption, the Delaunay complex $\delaunay(P(u))$ is
stable with respect to small perturbations. Namely, we can construct an open neighborhood $\tilde{O}_u\subset \R^n$
of $u$ such that the Delaunay complex is invariant in this neighborhood, i.e., $\delaunay(P(u))=\delaunay(P'(u'))$ for all
$u'\in \tilde{O}_u$. Therefore, by setting $O_{\alp}=O^\epsilon_u\cap \tilde{O}_u \cap U_{\alp}$, we can extend the decomposition of
the map $f$ as
\[
O_{\alp} \overset{g}{\longrightarrow} \R^{d} \overset{h}{\longrightarrow} \R^m.  
\]
In \cite{bdg}, the authors study explicit bounds on the perturbations of $u$ for
which the Delaunay complex $\delaunay(P)$ is invariant. 

\begin{rem}
  Precisely speaking, $h$ is defined on the image of $g$.
\end{rem}

\begin{rem}
When we consider $\ell$-th persistence diagram $\pdiag_\ell(\cX)$, it is
sufficient to deal only with $W_{\ell-1}$, $W_\ell$, and $W_{\ell+1}$. 
\end{rem}

\subsection{Smoothness of Persistence Map}

\subsubsection{Vietoris-Rips filtration $g$}

\begin{lem}
\label{lemma:differentiable_g_rips}
On the open set $O_{\vr}$, the map $g$ is of class $C^\infty$.
\end{lem}

\begin{proof}
For a simplex $\sigma=\{u_{i_0},\dots,u_{i_\ell}\}$, let $\{u_{i_a},u_{i_b}\}$ be the attaching edge, i.e.,
\[
r_\sigma=\frac{1}{2}||u_{i_b}-u_{i_a}||.
\]
From the assumption of the Vietoris-Rips general position, $r_\sigma$ is continuously differentiable on
$O_{\vr}$ whose entries in the $i_a$-th and the $i_b$-th coordinates are given by
\[
\frac{\partial r_\sigma}{\partial u_{i_a}}=\frac{1}{2}\frac{u_{i_a}-u_{i_b}}{||u_{i_b}-u_{i_a}||},~~~
\frac{\partial r_\sigma}{\partial u_{i_b}}=\frac{1}{2}\frac{u_{i_b}-u_{i_a}}{||u_{i_b}-u_{i_a}||}
\]
and zero otherwise. It is obvious from the same argument that $r_\sigma$ is continuously
differentiable arbitrarily many times.
\end{proof}

%

\begin{rem}
It follows from the proof of Lemma~\ref{lemma:differentiable_g_rips} that breaking the Vietoris-Rips
general position assumption immediately makes that map $g$ loose its differentiability. 
\end{rem}

%
%
\subsubsection{Alpha filtration $g$}

\begin{lem}\label{lemma:differentiable_g_alpha}
On the open set $O_\alp$, the map $g$ is of class $C^\infty$.
\end{lem}

\begin{proof}
Let $\sigma$ be a simplex in the alpha filtration $\alp(P)$ and $\tau$ be its attaching simplex. Then, it follows from the definition
of $\alp(P)$ that the birth radius $r_\sigma$ is given by the radius $\rho_\tau$ of the smallest circumsphere of $\tau$. 

Let us denote by $u_i=(u_{i,1},u_{i,2},u_{i,3})\in \R^3$ the coordinate of each point $u_i$. Let $u_{i,0}$ and
$u_{i,4}$ be given by $u_{i,0}=1$ and $u_{i,4}=\sum_{j=1}^3 u_{i,j}^2$, respectively. We define the
determinant
\[
M^{i_1i_2\dots i_k}_{j_1j_2\dots j_k}=\det
\left[\begin{array}{cccc}
u_{i_1,j_1} & u_{i_1,j_2} & \cdots & u_{i_1,j_k} \\
u_{i_2,j_1} & u_{i_2,j_2} & \cdots & u_{i_2,j_k} \\
\vdots & \vdots & \ddots & \vdots \\
u_{i_k,j_1} & u_{i_k,j_2} & \cdots & u_{i_k,j_k}
\end{array}\right].
\]
Then, the formulas for $\rho_\tau$ for $|\tau|=2,3,4$ are given in \cite{em} as
\begin{align*}
\tau&=\{u_i,u_j\}: \\
&\hspace{-0.3cm}\rho_{\tau}^2 =\frac{(M^{ij}_{10})^2 + (M^{ij}_{20})^2 + (M^{ij}_{30})^2}{4},\\
\tau&=\{u_i,u_j,u_k\}: \\
&\hspace{-0.3cm}\rho_{\tau}^2 =\frac{  \left(\sum_{z=1}^3(M^{ij}_{z0})^2\right) \cdot \left(\sum_{z=1}^3(M^{jk}_{z0})^2\right)
\cdot \left(\sum_{z=1}^3(M^{ki}_{z0})^2\right)  }{4\left(  (M^{ijk}_{230})^2 + (M^{ijk}_{130})^2 + (M^{ijk}_{120})^2 \right)},\\
\tau&=\{u_i,u_j,u_k,u_\ell\}:\\
&\hspace{-0.3cm}\rho_{\tau}^2 =\frac{(M^{ijk\ell}_{2340})^2\! +\! (M^{ijk\ell}_{1340})^2 \!+\! (M^{ijk\ell}_{1240})^2 \!+\!
4\cdot M^{ijk\ell}_{1230}\cdot M^{ijk\ell}_{1234}}{4\cdot (M^{ijk\ell}_{1230})^2}.
\end{align*}
It follows from the definition of the alpha general position and the above formula that $g$ is of class $C^\infty$ on $O_\alp$.
\end{proof}

\begin{rem}
\label{rem:jump}
It follows from the proof of Lemma~\ref{lemma:differentiable_g_alpha} that breaking the alpha general
position assumption immediately makes that map $g$ loose its differentiability.
\end{rem}

\subsubsection{The map $h$}
\label{sec:map_h}

Let us next study the map $h$ and its differentiability.  
The map $h$ can be computed by Algorithm~\ref{alg:persistence} below, which consists of three parts and is
based on \cite{bkr} and \cite{zc}. Each procedure is presented in what follows.

\begin{algorithm}[h!]
  \caption{Compute Persistence Data from $W$ and $\partial$}
  \label{alg:persistence}
  \begin{algorithmic}
    \Procedure{ComputePersistenceData}{$W, \partial$}
      \State $B$ $\leftarrow$ \Call{BoundaryMatrix}{$W, \partial$}
      \State $P$ $\leftarrow$ \Call{PersistenceLeftRight}{$B$}
      \State \Return \Call{PersistenceData}{$P$}
    \EndProcedure
  \end{algorithmic}
\end{algorithm}

\paragraph{\textproc{BoundaryMatrix}}
For a matrix $A = (A_{ij})$, let us denote its $j$-th column of $A$ by $A_j$, and for a non-zero column
$A_j$, set $\textrm{pivot}(A_j) := \max \{ i \mid A_{ij} \neq 0 \}$, 
called the pivot index.

Let us order the set $W$ of all simplices, $\sigma_1 < \sigma_2 < \cdots < \sigma_K$,
by the lexicographical order of $(r_{\sigma}, \dim \sigma) \in \realpositive \times \Z$. 
If two (or more) simplices $\sigma,\sigma'$ appear at the same birth radius with the same dimension,
we order them by an appropriate rule. 

In this order, a subset $\{\sigma_1, \cdots, \sigma_k\}$ for any $k$ is a subcomplex of $W$ in both the
Vietoris-Rips and the alpha filtrations.

Let $C$ be the vector space spanned by
$\{\sigma_1, \ldots, \sigma_K\}$.
A matrix representation $B = (B_{ij})$ of the boundary map $\partial : C \to C$ is constructed in such a way that 
for an $\ell$-simplex $\sigma_i=\{u_{i_0},\dots, u_{i_\ell}\}$ with $i_0<\dots<i_\ell$, the $(i,j)$-entry is given by
\[
  B_{ij} = 
  \left\{\begin{array}{ll}
           (-1)^k, & \sigma_j=\{u_{i_0},\dots,\widehat{u_{i_k}},\dots,u_{i_\ell}\} ,\\
           0, & {\rm otherwise}.
         \end{array}\right.
\]

\paragraph{\textproc{PersistenceLeftRight}}

A column operation of the form $A_j \leftarrow A_j +\lambda A_k$
is called 
a left-to-right operation if $k < j$. We call a matrix $A'$ \textit{derived} from $A$ if 
$A'$ can be transformed from $A$ by left-to-right operations.
We call a matrix $A'$ {\em reduced} if no two non-zero columns have the same pivot index.
If a reduced matrix $A'$ is derived from $A$, we call it a \textit{reduction} of
$A$. In this case, we define
\begin{align*}
  P_{A'} &= \{ (i,j) \mid A'_j \neq 0 \mbox{ and } i = \textrm{pivot}(A'_j) \},\\
  E_{A'} &= \{ i \mid 1 \leq i \leq K, i\neq i' \mbox{ and } i\neq j' \mbox{ for }\forall 
(i',j') \in P_{A'}\}
\end{align*}

For a reduction $B'$ of the matrix representation $B$, it follows from \cite{zc} that we can find a basis
$\{w_1, \ldots, w_K\}$ of $C$ satisfying: (i) the subspace spanned by $\{\sigma_1, \ldots, \sigma_k\}$
is equal to the subspace spanned by $\{w_1,\ldots, w_k\}$ for any $k$, and
(ii) $\partial w_j$  is given by
\[
  \partial w_j = \left\{
  \begin{array}{ll}
    w_i, & (i,j) \in P_{B'} \\
    0, & \mbox{otherwise}
  \end{array}
  \right. .
\]

Algorithm~\ref{alg:persistence_bkr} (a modification
of Algorithm~1 in \cite{bkr})
shows a simple algorithm to compute
$P_{B'}$ of the matrix $B$.
Note that $E_{B'}$ is easily computable from $P_{B'}$.
The algorithm processes columns
from left to right; for each column, other columns are added from the left
until a new pivot index appears or the column becomes zero. 

\begin{algorithm}[h!]
  \caption{Reduction Algorithm}
  \label{alg:persistence_bkr}
  \begin{algorithmic}
    \Procedure {PersistenceLeftRight}{$B$}
      \State $A\leftarrow B$; $L\leftarrow[0,\ldots,0]$ 
      \For{$j=1,\ldots,K$}
        \While{$A_j\neq 0$ and $L[\pivot(A_j)]\neq 0$}
          \State $ k \leftarrow \pivot(A_j)$
          \State $ l \leftarrow L[\pivot(A_j)]$
          \State $A_j\leftarrow A_j - a_{kj} a_{kl}^{-1} \cdot A_{l}$
        \EndWhile
        \If{$A_j\neq 0$} 
          \State $i\leftarrow \pivot(A_j)$
          \State $L[i]\leftarrow j$
        \EndIf
      \EndFor
      \State {\bf return} $\{(i,j) \mid L[i]=j \mbox{ and } j \neq 0\}$ 
    \EndProcedure
  \end{algorithmic}
\end{algorithm}

\paragraph{\textproc{PersistenceData}}
The basis $\{w_1, \ldots, w_K\}$ represents the decomposition of the persistent homology, and hence,
we obtain the persistence data from $P_{B'}$ as follows:
%
\[
\begin{cases}
(r_{\sigma_i}, r_{\sigma_j}), & \text{for~} (i,j) \in P_{B'} \text{~and~} \dim \sigma_i = \ell, \\
(r_{\sigma_i}, \infty),            & \text{for~} i \in E_{B'} \text{~and~} \dim \sigma_i = \ell.
\end{cases}
\]
Note that the persistence data is independent of the choice of algorithm from the unique decomposition of 
persistent homology. 

As a result of the above algorithms, the map $h : \R^d\rightarrow \R^m$ is expressed by
\[
h(r) = \left[\begin{array}{c}
h_{{\rm fin}}(r)\\
h_{{\rm inf}}(r)
\end{array}\right]
\]
with
\[
h_{{\rm fin}}(r)=\left[\begin{array}{c}
r_{\sigma_{i_1}}\\
r_{\sigma_{j_1}}\\
\vdots\\
r_{\sigma_{i_s}}\\
r_{\sigma_{j_s}}
\end{array}\right],
\quad
h_{{\rm inf}}(r)=
\left[\begin{array}{c}
r_{\sigma_{i_{s+1}}}\\
\vdots\\
r_{\sigma_{i_{s+t}}}
\end{array}\right],
\]
where 
$\{(i_1, j_1), \ldots, (i_s, j_s)\} = \{(i,j) \in P_R \mid \dim\sigma_i=\ell \mbox{ and } (r_{\sigma_i}, r_{\sigma_j}) \in U_\epsilon\}$, $U_\epsilon=\{p\in \bar{\R}^2\mid \dinfty(p,\Delta)\geq \epsilon\}$, and $\{i_{s+1}, \cdots, i_{s+t}\} = \{i \in E_R\mid \dim\sigma_i = \ell \}$.  
Hence, we have $m=2s+t$.
The condition $(r_{\sigma_i}, r_{\sigma_j}) \in U_\epsilon$
guarantees the uniqueness of $m$ from Lemma \ref{lemma:nondecreasing}.
%
%


The following lemma is derived easily from the explicit form of $h(r)$.
\begin{lem}
\label{lemma:differentiable_h}
The map $h$ is of class $C^\infty$ on $\R^d$.
\end{lem}

We remark that it is sufficient to treat the $\ell-1$, $\ell$, and the $(\ell+1)$-simplices in
Algorithm~\ref{alg:persistence_bkr}, if we want the persistence diagram $D_\ell$ for a
single value of $\ell$. 

\subsubsection{The map $f$}
It follows from the chain rule
$
Df(u)=Dh(r)\circ Dg(u)
$
that the explicit form of the derivative $Df(u)$ is given by

\[
Df(u) = \left[\begin{array}{c}
Df_{{\rm fin}}(u)\\
Df_{{\rm inf}}(u), 
\end{array}\right],
\]
with
\[
Df_{{\rm fin}}(u) =\left[\begin{array}{c}
\frac{\partial r_{\sigma_{i_1}}}{\partial u}\\
\frac{\partial r_{\sigma_{j_1}}}{\partial u}\\
\vdots\\
\frac{\partial r_{\sigma_{i_s}}}{\partial u}\\
\frac{\partial r_{\sigma_{j_s}}}{\partial u}
\end{array}\right],
\quad
Df_{{\rm inf}}(u)=\left[\begin{array}{c}
\frac{\partial r_{\sigma_{i_{s+1}}}}{\partial u}\\
\vdots\\
\frac{\partial r_{\sigma_{i_{s+t}}}}{\partial u}\\
\end{array}\right].
\]

%
\begin{prop}
The map $f$ is of class $C^\infty$ on $O_{\vr}$ and $O_\alp$. Moreover, the derivatives are
independent of the choice of algorithms up to permutations of coordinates in $\R^m$.
\end{prop}
\begin{proof}
The first statement follows from the chain rule and Lemmas~\ref{lemma:differentiable_g_rips},
\ref{lemma:differentiable_g_alpha}, and \ref{lemma:differentiable_h}.
For the second statement, let us assume two different expressions $h$ and $h'$.
From the uniqueness of the persistence data, we can express $h$ and $h'$ by appropriate
permutations if necessary as
\begin{align*}
&h(u) = \left[\begin{array}{c}
h_{{\rm fin}}(u)\\
h_{{\rm inf}}(u)
\end{array}\right], \\
&h_{{\rm fin}}(u)=\left[\begin{array}{c}
r_{\sigma_{i_1}}\\
r_{\sigma_{j_1}}\\
\vdots\\
r_{\sigma_{i_s}}\\
r_{\sigma_{j_s}}
\end{array}\right], ~
h_{{\rm inf}}(u)=
\left[\begin{array}{c}
r_{\sigma_{i_{s+1}}}\\
\vdots\\
r_{\sigma_{i_{s+t}}}
\end{array}\right],
\end{align*}
and
\begin{align*}
&h'(u) = \left[\begin{array}{c}
h'_{{\rm fin}}(u)\\
h'_{{\rm inf}}(u)
\end{array}\right],\\
&h'_{{\rm fin}}(u)=\left[\begin{array}{c}
r_{\sigma_{i'_1}}\\
r_{\sigma_{j'_1}}\\
\vdots\\
r_{\sigma_{i'_s}}\\
r_{\sigma_{j'_s}}
\end{array}\right], ~
h'_{{\rm inf}}(u)=
\left[\begin{array}{c}
r_{\sigma_{i'_{s+1}}}\\
\vdots\\
r_{\sigma_{i'_{s+t}}}
\end{array}\right]
\end{align*}
with
\[
r_{\sigma_{i_k}}=r_{\sigma_{i'_k}},~~r_{\sigma_{j_k}}=r_{\sigma_{j'_k}}
\]
for all $k$. On the other hand, it follows from the definitions of the Vietoris-Rips and the alpha general
positions of $u$ that there uniquely exist the attaching simplices $\eta_k$ and $\xi_k$ such that
\[
r_{\eta_k}=r_{\sigma_{i_k}}=r_{\sigma_{i'_k}},~~r_{\xi_k}=r_{\sigma_{j_k}}=r_{\sigma_{j'_k}}
\]
for each $k$. Hence, this leads to
\[
\frac{\partial r_{\eta_k}}{\partial u} = 	\frac{\partial r_{\sigma_{i_k}}}{\partial u} = 	\frac{\partial r_{{\sigma_{i'_k}}}}{\partial u},~~
\frac{\partial r_{\xi_k}}{\partial u} = 	\frac{\partial r_{\sigma_{j_k}}}{\partial u} = 	\frac{\partial r_{\sigma_{j'_k}}}{\partial u},	
\]
and completes the proof of the second statement.
\end{proof}

\begin{rem}
Since (\ref{eq:stability1}) and (\ref{eq:stability2}) are $1$-Lipschitz, it is reasonable to have differentiability
from Rademacher's theorem \cite{rademacher}. The discussion in this section provides a constructive proof
of this fact, and furthermore, shows that the derivatives are independent of the choice of algorithm.
\end{rem}

\section{Continuation}
\label{sec:cont}

\subsection{Continuation by Newton-Raphson Method}
\label{sec:newton_square}

We first recall the standard Newton-Raphson continuation method \cite{continuation}.
Let $U$ be an open set in $\R^n$ and $\varphi : U\times \R \to \R^n$ be a $C^1$ mapping.
Suppose that $(\bar{u},\bar{\lambda}) \in U\times \R $ satisfies $\varphi(\bar{u}, \bar{\lambda}) = 0$.
Our purpose is to solve $\varphi(u, \lambda) = 0$ with respect to $u \in U$
for a given $\lambda$.
The existence and the local uniqueness of the
solution $u = u_{\lambda}$
is guaranteed by the implicit function theorem
when $D_u \varphi (\bar{u}, \bar{\lambda})$ is regular
and $\lambda$ is sufficiently close to $\bar{\lambda}$. 

In practical computations, we
find the solution $u_{\lambda}$ for each $\lambda$
by iteratively solving linear equations as follows.
By taking an appropriate initial point $u^{(0)}$, 
the linear approximation of $\varphi(u,\lambda)$ at each iteration step $j$ is given by
\[
\varphi(u, \lambda)\approx\varphi(u^{(j)}, \lambda)+D_u\varphi(u^{(j)},\lambda)(u-u^{(j)}).
\]
Setting the right hand side to be zero
\[
\varphi(u^{(j)}, \lambda)+D_u\varphi(u^{(j)}, \lambda)(u-u^{(j)})=0,
\]
we obtain an update of the approximate solution
\begin{equation}
\label{eq:newton_square}
u^{(j+1)}=u^{(j)}-D_u\varphi(u^{(j)}, \lambda)^{-1}\varphi(u^{(j)}, \lambda).
\end{equation}

This iteration method is called the Newton-Raphson Method, and the convergence 
of the iterations under suitable regularity of derivatives is well studied (e.g., \cite{or}).

The continuation of the solution $(\bar{u},\bar{\lambda})$ to a parameter $\lambda'$
is achieved by gradually changing the parameter $\lambda$.
That is, for $\lambda_0=\bar{\lambda}< \lambda_1< \cdots< \lambda_N=\lambda'$,
the Newton-Raphson method is applied for each parameter $\lambda_i$ 
by setting $u^{(0)} = u_{\lambda_{i-1}}$ with $u_{\lambda_0}  = \bar{u}$ 

\subsection{Newton-Raphson Method by Pseudo-Inverse}
\label{sec:newton_pseudo_inverse}

Let $f: U\rightarrow \R^m$ be a persistence map defined on an open set $U\subset\R^n$. 
Let us define a map
\begin{equation}\label{eq:F}
	F:U\times\R^m\rightarrow \R^m
\end{equation}
by $F(u,v)=f(u)-v$. For a given pair $(\bar{u},\bar{v})$ satisfying $F(\bar{u},\bar{v})=0$ and $v$ close to $\bar{v}$,
our interest is to solve $F(u,v)=0$ with respect to $u\in U$. 
The existence of the solution is guaranteed when $Df(\bar{u})$ is 
surjective and $v$ is sufficiently close to $\bar{v}$. 
Hence, for the rest of the paper, we add the assumption that $m \leq n$. 
In this case, the solutions form an $n-m$ dimensional manifold.

The basic strategy to solve $F(u,v)=0$ is the same as in Section~\ref{sec:newton_square}, and
we derive an iteration method for $u^{(j)}$, $j=0,1,\dots, N$, with $u^{(0)}=\bar{u}$ converging to
the solution. Namely, the linear approximation of $F$ at $u^{(j)}$ leads to
\begin{equation}\label{eq:linearization}
	Df(u^{(j)})(u-u^{(j)})+F(u^{(j)},v)=0,
\end{equation}
where $Df(u^{(j)})=D_uF(u^{(j)},v)$, and we derive an update $u^{(j+1)}$ by solving the linear equation
with respect to $u$. However, we note that the linear equation (\ref{eq:linearization}) is defined having
domain and image with different dimensions in general. Thus $Df(u^{(j)})$ is an $m\times n$ rectangular matrix with $m \leq n$. 

In general, let us consider a linear equation
\begin{equation}\label{eq:lineareq}
	Ax=b,~~~A\in M_{m,n}(\R),~b\in \R^m
\end{equation}
with $m\leq n$. For solving this type of linear equations, we first recall the concept of pseudo-inverse and
explain its relation to solutions of the linear system (\ref{eq:lineareq}). 

For a matrix $A\in M_{m,n}(\R)$, there exists a unique matrix $X\in M_{n,m}(\R)$ satisfying the so-called Penrose equations:
%
\begin{align*}
  &AXA=A,~~~XAX=X, \\
  &(AX)^T=AX,~~~(XA)^T=XA,
\end{align*}
where $A^T$ is the transpose matrix of $A$.
The unique matrix solution $X$ is called the {\em pseudo-inverse} of $A$ and denoted by $A^\dagger$. 
%

An explicit formula to construct $A^\dagger$ is given for instance in \cite{bg,hj}. Assume that the matrix
$A \in M_{m,n}(\R)$ has the rank $k \le m$. Then, it has a singular value decomposition (SVD) of the form
$A=V \Sigma W^T$, where $V$ and $W$ are orthogonal matrices, and the matrix
$\Sigma = (\sigma_{i,j}) \in M_{m,n}(\R)$ has $\sigma_{i,j}=0$ for all $i \ne j$, and
$\sigma_{1,1} \ge \sigma_{2,2} \ge \cdots \ge \sigma_{k,k} > \sigma_{k+1,k+1}=\cdots =\sigma_{m,m}=0$.
The numbers $\sigma_i := \sigma_{i,i}$, $i=1,\dots,m$, are called the {\em singular values} of the matrix $A$. 
From the SVD of the matrix $A$, the pseudo-inverse $A^\dagger$ can be obtained by the formula
%
\[
A^\dagger = W \Sigma^\dagger V^T, 
\]
%
where $\Sigma^\dagger=(\sigma^\dagger_{i,j})$ is the $n\times m$ matrix with $\sigma^\dagger_{i,j}=0$ for all $i\neq j$, $\sigma^\dagger_{i,i}=\sigma_i^{-1}$ for $i=1,\dots,k$, and $\sigma^\dagger_{i,i}=0$ for $k<i\leq m$. 
%
%

%
%

Equation (\ref{eq:lineareq}) has a solution for $x$ if and only if $b\in \image (A)$. In such a case,
there is a unique minimum norm solution $x$ of (\ref{eq:lineareq}), meaning that the Euclidean norm
$||x||$ is minimum among all the solutions of (\ref{eq:lineareq}). If $b\notin \image (A)$, then a
least-squares solution $x$ of (\ref{eq:lineareq}) is a vector minimizing the error $||Ax-b||$. The following
proposition provides us with relations between the pseudo-inverse and solutions of the linear equation.
For a proof, the reader may refer to \cite{bg} for instance.
\begin{prop}
Let $A\in M_{m,n}(\R)$ and $b\in \R^m$. If $b\in\image (A)$, then the unique minimum norm solution
of $Ax=b$ is given by $x=A^\dagger b$. If $b\notin\image (A)$, among the least-squares solutions of $Ax=b$,
$A^\dagger b$ is the one of minimum norm.
\end{prop}

This proposition provides us with a method for finding the minimum norm solution of the
equation (\ref{eq:linearization}). Namely, we update the approximate solution $u^{(j)}$ by
\begin{equation}
\label{eq:update}
u^{(j+1)}=u^{(j)}-A_j^\dagger F(u^{(j)},v),
\end{equation}
where $A_j=Df(u^{(j)})$. Note that, from the minimality condition on the norm, the update $u^{(j+1)}$ is chosen to be closest to
$u^{(j)}$, and this is a natural choice for the purpose of continuation.

The convergence of the iterations (\ref{eq:update}) is studied in \cite{b2} (see also \cite{b1}), and we summarize
it as follows.
\begin{prop}
Let $r>0$ be a positive real number such that $f\in C^1(B_r(\bar{u}))$. Let $\alpha,\beta$ be positive constants
such that, for all $u,w\in B_r(\bar{u})$ with $u-w\in \image Df(w)^T$, the followings hold:
\begin{align*}
&||Df(w)(u-w)-f(u)+f(w)||\leq \alpha||u-w||,\\
&||(Df(w)^\dagger-Df(u)^\dagger)f(u)||\leq \beta||u-w||,\\
&\alpha||Df(z)^\dagger||+\beta=\gamma<1~~~{\rm for~all~}z\in B_r(\bar{u}),\\
&||Df(\bar{u})^\dagger||\cdot||f(\bar{u})||<(1-\gamma)r.
\end{align*}
Then the iteration {\rm (\ref{eq:update})} converges to a solution of
\[
Df(u)^TF(u,v)=0
\]
which lies in $B_r(\bar{u})$.
\end{prop}

When $m=n$, this proposition provides a criterion for the convergence of the Newton-Raphson
method (\ref{eq:newton_square}).

For $\rank (Df(u))=m$, the iteration converges to a solution of $F(u,v)=0$. On the other hand, when
$\rank (Df(u))<m$, the convergent point $u$ does not necessarily satisfy $F(u,v)=0$, but only implies
$F(u,v)\in \kernel Df(u)$. 
Thus, in our continuation method, we suppose that all the singular values $\sigma_1,\dots,\sigma_m$ are positive. 

\subsection{Continuation of Point Clouds}
\label{sec:continuation}

We use the iteration (\ref{eq:update}) for continuations of a point cloud. Let $f:U\rightarrow \R^m$
be a persistence map.  Suppose that $(u_s,v_s)\in U\times\R^m$ is a pair satisfying $f(u_s)=v_s$.
This pair can be regarded as the initial point of the continuation. Our task is to continuate it to a
target persistence data $v_t\in\R^m$ and obtain $u_t\in \R^n$ satisfying $f(u_t)=v_t$. 

As the simplest way of the continuation,
we divide the line between $v_s$ and $v_t$ equally into $N$ small
segments, and for each $v=v_s+k\Delta v$, $k=1,\dots,N$, where $\Delta v=\frac{v_t-v_s}{N}$,
we apply the iteration method (\ref{eq:update}) and obtain the point cloud $u$ satisfying $f(u)=v$.
This process is summarized in Algorithm~\ref{alg:persistence_continuation}.

\begin{algorithm}[h!]
  \caption{Continuation of a point cloud}
  \label{alg:persistence_continuation}
  \begin{algorithmic}
    \State{\bf input} $u_s,v_s,v_t,N$
    \State $\Delta v=(v_t-v_s)/N$,  $u= u_s$
    \For {$k=1:N$} 
      \State $v=v_s+k\Delta v$
      \State \textbf{if} {\textproc{PseudoInverseNewton}($u,v$) converges}
        \State \ \ \ \ $u'=$\textproc{PseudoInverseNewton}($u,v$)
        \State \ \ \ \ $u=u'$
      \State \textbf{else}
        \State \ \ \ \ {\bf break}
      \State \textbf{endif}
    \EndFor
    \State{\bf output}	 $u_t=u$
  \end{algorithmic}
\end{algorithm}

We can, of course, adaptively choose the length of $\Delta v$ at each continuation step. Furthermore,
we can also adopt an appropriate curve connecting $v_s$ to $v_t$ if necessary. 

We also note that the image of $f$ is not generally equal to the target space $\R^m$.
Hence, if we choose $v_t$ in $\R^m \setminus \image (f)$, the continuation fails at some
$v=v_s+k\Delta v$. See Section \ref{ex:existence_regions} for such an example.

%
It is often the case in practical problems that the point clouds need to satisfy some constraints $g_i(u)=0$, $i=1,\dots,r$ (e.g., conservation laws in mechanics). In such a case, we need to solve the continuation under these constraints, and  its modification is straightforward. Namely, we first extend the original setting (\ref{eq:F}) to
\[
	\tilde{F}:U\times \R^m\rightarrow\R^{m+r}
\]
by $\tilde{F}(u,v)=(f(u)-v,g_1(u),\dots,g_r(u))$. 
Then, by replacing $A_j$ and $F$ in (\ref{eq:update}) with $\tilde{A}_j = D_u\tilde{F}$ and $\tilde{F}$,
respectively, we obtain the appropriate formulas for continuation of point clouds under constraints $g_i(u)=0$, $i=1,\dots,r$. 



%
%
\subsection{Symmetry}
\label{sec:symmetry}

It should be noted that we need to remove symmetries induced by translations and rotations in order
to isolate a solution $u$ of $f(u)=v$. For a point cloud $P$ in $\R^3$, the following restrictions will
remove these symmetries:
\begin{enumerate}
\item[(i)] Fix one of the vertices in $P$ at the origin of $\R^3$, say $u_0\equiv(0,0,0)$.
\item[(ii)] Select one of the vertices in $P\setminus\{u_0\}$, say $u_1$. This vertex is supposed to move
only on the line defined by $\overrightarrow{u_0u_1}$ during the continuation.
\item[(iii)] Select one of the vertices in $P\setminus\{u_0,u_1\}$, say $u_2$. This vertex is supposed to
move only on the plane defined by $\overrightarrow{u_0u_1}, \overrightarrow{u_0u_2}$ during the continuation.
\end{enumerate}
The first restriction eliminates translation symmetries and the second and third restriction eliminate
rotation symmetries.

In practice we choose our basis of the coordinate system $(x,y,z)\in\R^3$ in such a way that,
in addition to $u_0$ being fixed at the origin, we have that $u_1$
stays on the $x$-axis and $u_2$ stays on the $(x,y)$-plane during the continuation.
Hence, in the following, let us redefine $n=3M-6$ and
\begin{equation}
\label{eq:symmetry_restriction}
u=(u_1, u_2, u_3, \dots, u_M),
\end{equation}
where $u_1=x_1, u_2=(x_2,y_2), u_i=(x_i,y_i,z_i)$ for $3\leq i\leq M$.

For a general point cloud in $\R^L$, these restriction to eliminate the symmetries
should be appropriately modified.  





\subsection{Non-convergence and Zero Singular Values}\label{sec:zsv}
The Newton-Raphson method by pseudo-inverse does not work well when the matrix
$A_j$ in (\ref{eq:update}) has a singular value close to zero. In this section, we discuss those cases in the alpha and Vietoris-Rips filtrations. 

\subsubsection{For the case of alpha filtrations}
For the case of alpha filtrations, we show that
a singular value close to zero appears when there is a birth-death
pair $(b, d)$ with $b \approx d$ in the persistence diagram
$\pdiag_{\ell}(\alp(P))$.
Here, we impose the natural assumptions
that a point cloud $P$ in $\R^L$ ($L=2,3)$
satisfies the general position assumption and
consists of at least $L+1$ points.

Let $\alp(P, \alpha+)$ be the alpha complex
constructed by simplices whose birth radii are smaller than or equal to $\alpha$.
Notice that this is different from $\alp(P, \alpha)$, where
$\alpha$ is equal to the birth radius of a simplex.

\begin{prop}\label{prop:birth_death_pair}
 Let $P$ be a point cloud and let
 $(b,d)$ be a birth-death pair in $\pdiag_\ell(\alp(P))\ (\ell=1,2)$.
 If the birth radius of any simplex
 is not contained in the open interval $(b, d)$,
 there exist an attaching $\ell$-simplex $\sigma$ and
 an attaching $(\ell+1)$-simplex $\tau$ 
 such that $r_{\sigma} = b$, $r_{\tau} = d$, and $\tau \supset \sigma$.
\end{prop}

To prove the proposition, we need the following lemma.

\begin{lem}\label{lemma:attached_simplex}
  Let $\sigma$ be a simplex and 
  $\tau$ be its attaching simplex with $\sigma\neq\tau$.
  Then,
  the inclusion from $|\alp(P, r_{\tau})|$ to $|\alp(P, r_{\tau}+)|$
  is a deformation retract. In particular, $r_{\tau}$ 
  is neither a birth time nor a death time.
\end{lem}

To prove the lemma, we recall some properties
about Voronoi decomposition \cite{dcvo}. 
For $k+1$ points $\{u_0, \ldots, u_k\}$ in $P$,
let $V(u_0, \cdots, u_k)$ be defined as
\begin{align*}
  &V(u_0, \cdots, u_k) \\
  =& \{ x \in \R^L \mid d(x, u_0) = \cdots = d(x, u_k), \mbox{ and } \\
  & \hspace{4em} d(x, u_0) \leq d(x, u) 
    \mbox{ for } \forall u \in P \backslash \{u_0, \ldots, u_k\}\}.
\end{align*}
For $k=0$, the set $V(u_0)$ is called a Voronoi cell.
From the theory of Voronoi decompositions and Delaunay triangulations,
we have the following facts:
\begin{enumerate}
\item $V(u_0, \cdots, u_k) = V(u_0) \cap \cdots \cap V(u_k)$.
\item $\{u_0, \cdots, u_k\}$ forms a Delaunay $k$-simplex
  if and only if $V(u_0, \cdots, u_k)$ is
  not empty.
  \label{enum:delaunay_voronoi}
\item $V(u_0, \cdots, u_k)$ is closed, convex, and
  contained in a $(L-k)$-dimensional affine subspace
  $W(u_0, \cdots, u_k) = 
  \{ x \in \R^L \mid d(x,u_0) = \cdots = d(x, u_k)\}$.
  $W(u_0, \cdots, u_k)$ is orthogonal to
  the $k$-dimensional affine subspace spanned by
  $\{u_0, \ldots, u_k\}$.
  \label{enum:voronoi_1}
\item The boundary of $V(u_0, \cdots, u_k)$ 
  relative to $W(u_0, \cdots, u_k)$ is
  $\bigcup_{v \in Y(u_0, \cdots, u_k)} V(u_0, \cdots, u_k, v)$,
  where 
  $Y(u_0, \cdots, u_k) = \{v \in P\backslash\{u_0, \cdots, u_k\} \mid V(u_0, \cdots, u_k) \cap V(v) \neq \emptyset \}$.
  \label{enum:voronoi_2}
\item If $V(u_0, \cdots, u_k)$ is not empty,
  $Y(u_0, \cdots, u_k)$ is not empty.
  \label{enum:voronoi_3}
\item A Delaunay $k$-simplex $\{u_0, \ldots, u_k\}$ is attaching
  if and only if this simplex and
  $V(u_0, \cdots, u_k)$ have a non-empty intersection.
  \label{enum:delaunay_attaching}
\end{enumerate}

\begin{proof}(Lemma \ref{lemma:attached_simplex})
  First we consider the case of $L=2, \dim \sigma = 1$, and
  $\dim \tau = 2$.
  Let $u_1, u_2$ be the endpoints of 
  $\sigma$ and $u_0$ be the other vertex of 
  $\tau$.
  From the facts \ref{enum:delaunay_voronoi} and \ref{enum:voronoi_1} above, 
  $V(u_1, u_2)$ is not empty and 
  contained in the perpendicular bisector of $\sigma$.
  From the fact \ref{enum:voronoi_3}, $Y(u_1, u_2)$ is not empty and
  $V(u_1, u_2)$ is 
  a line with one endpoint
  or with two endpoints.
  Since $\sigma$ is not attaching,
  $V(u_1, u_2) \cap \sigma = \emptyset$ from
  the fact \ref{enum:delaunay_attaching}.
  Let $E$ be the endpoint of $V(u_1, u_2)$ close to $\sigma$, which is given by $E = V(u_1, u_2, u_0)$ from
  the fact \ref{enum:voronoi_2}.
  Then, from the definition of $E = V(u_1, u_2, u_0)$
  and the general position assumption,
  the following holds:
  \begin{align*}
    d(E, u_1) &= d(E, u_2) = d(E, u_0), \\
    d(E, u) &> d(E, u_1) \mbox{ for } \forall u \in P\backslash\{u_0, u_1, u_2\}.
  \end{align*}

  We take a new orthogonal coordinate system 
  on $\R^2$ satisfying
  \begin{itemize}
  \item $E = (0, 0)$ 
  \item $V(u_1, u_2)$ is contained in $\{ (t,0)\mid t \leq 0 \}$.
  \end{itemize}
  In this coordinate, $u_0$, $u_1$, and $u_2$, are described as
  \begin{align*}
    u_0 &= (r \cos \theta_0, r \sin \theta_0), \\
    u_1 &= (r \cos \theta_1, r \sin \theta_1), \\
    u_2 &= (r \cos \theta_1, -r \sin \theta_1), 
  \end{align*}
  since $u_0, u_1, u_2$ lie on the same circle whose
  center is $E$ and $\sigma$ is orthogonal to 
  $V(u_1, u_2)$
  (Figure~\ref{fig:coordinate_for_attached_simplex}).
  It follows from $\sigma \cap V(u_1, u_2) = \emptyset$ that
  $\theta_1$ must be contained in $[0, \pi/2)$.
  Since $E_\epsilon = (-\epsilon, 0) \in V(u_1, u_2) \backslash V(u_1, u_2, u_0)$ 
  for small $\epsilon > 0$,
  we have $d(E_\epsilon, u_1) < d(E_\epsilon, u_0)$ and
  \begin{align*}
    0 < &\ d(E_\epsilon, u_0)^2 - d(E_\epsilon, u_1)^2  \\
    = &\left((r \cos \theta_0 + \epsilon)^2 + (r \sin \theta_0)^2\right) \\
    - &\left((r \cos \theta_1 + \epsilon)^2 + (r \sin \theta_1)^2 \right)\\
    = &2r\epsilon(\cos \theta_0 - \cos\theta_1).
  \end{align*}
  Hence $ -\pi/2 < -\theta_1 < \theta_0  < \theta_1 < \pi/2$ holds and
  $\tau$ is an obtuse triangle whose longest edge
  is $\sigma$. 
  Therefore, for $\sigma' = \overline{u_0 u_1}$ and
  $\sigma'' = \overline{u_0 u_2}$,
  we can show that
  $r_{\sigma'} < r_{\tau}$ and $r_{\sigma''} < r_{\tau}$
  from the general position assumption and the definition of the birth radius of a simplex
  $\omega$
  \begin{align*}
    r_{\omega} = \inf \{& \mbox{the radius of } B \mid \ B \mbox{ is an
      empty ball} \\
    &\ \ \ \mbox{satisfying } \partial B \supset (\mbox{vertices of }
    \omega)\}.
  \end{align*}

  \begin{figure}[ht]
    \centering
    \includegraphics[width=0.7\hsize]{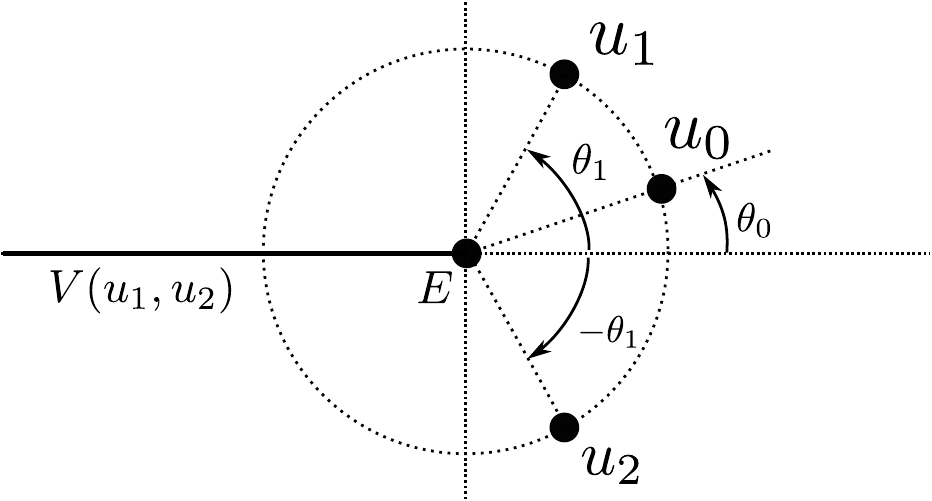}
    \caption{Orthogonal coordinate system for an attaching 2-simplex 
      $\tau = \{ u_0, u_1, u_2\}$.
    }
    \label{fig:coordinate_for_attached_simplex}
  \end{figure}
  
  Since any other birth radii are different from $r_{\tau}$ by
  the general position assumption, we have
  $\alp(P, r_{\tau}+) \backslash \alp(P, r_{\tau}) = \{\sigma, \tau\}$.
  Therefore, the inclusion
  \begin{equation*}
    |\{\sigma', \sigma'', u_0, u_1, u_2\}| 
    \hookrightarrow{} 
    |\{\tau, \sigma, \sigma', \sigma'', u_0, u_1, u_2, \}| ,
  \end{equation*}
  leads to the desired deformation retract.
  
  The case of $L=3, \dim\sigma=1$, and $\dim\tau = 2$ 
  can be proved in the same way by considering
  the plane that contains three vertices of $\tau$.
  
  The case of $L=3, \dim\sigma=2$, and $\dim\tau = 3$
  can also be proved in a similar way by replacing
  $V(u_1, u_2)$ in the above argument
  to $V(u_1, u_2, u_3)$.
  In this case, the 3-simplex $\tau$ is given by
  a tetrahedron whose four vertices are on a 
  half side of the circumsphere of the tetrahedron.
\end{proof}


The following corollary is a straightforward consequence
of Lemma~\ref{lemma:attached_simplex}.

\begin{cor}\label{cor:birth_death}
  For a simplex $\tau$, if $r_{\tau}$ is either a
  birth time or a death time in the persistence diagram,
  $\tau$ is an attaching simplex and
  does not attach any faces of $\tau$.
\end{cor}

  

\begin{proof} (Proposition \ref{prop:birth_death_pair})
 Since $d$ is a death time in $\pdiag_{\ell}(\alp(P))$,
 an $(\ell+1)$-simplex $\tau$ is born at time $d$, hence
 $d = r_{\tau}$.
 From Corollary \ref{cor:birth_death},
 $\tau$ is an attaching simplex and does not attach any faces of $\tau$.
 With the general position assumption,
 $\tau$ is the unique simplex satisfying
 $d = r_{\tau}$.
 
 Next, we consider the change at time $b$.
 Since the birth radius of any simplex 
 is not contained in $(b, d)$ and $\tau$ is unique
 as above, 
 $\alp(P, \alpha) = \alp(P, r_{\tau}+) \backslash \{\tau\}$
 for any $\alpha \in (b, d]$.
 Therefore,
 one of the $\ell$-dimensional faces $\sigma$ of $\tau$ appears at
 time $b$. Otherwise, the generator by
 $\partial \tau$
 keeps unchanged at time $b$ and
 this contradicts the fact that $b$ is the birth time of the pair $(b,d)$.
The simplex $\sigma$ is an attaching simplex from Corollary \ref{cor:birth_death}, and hence this concludes the proof.
\end{proof}

Now, we show that if a birth-death pair $(b, d)$ is close enough $b \approx d$ for the birth radius of any simplex not to be contained in $(b,d)$, then a singular value close to zero appears in the derivative of the persistence map. 
In this case, from Proposition \ref{prop:birth_death_pair},
there exist an $(\ell+1)$-simplex $\tau$ and its face $\sigma$ satisfying
$r_{\tau} = d$ and $r_{\sigma} = b$.
Let $p_1, \ldots, p_{\ell}$ be the vertices of $\sigma$ and
$p_0, p_1, \ldots, p_{\ell}$ be the vertices of $\tau$.
Let $c_0$ be the center of the minimal circumsphere of $\tau$ and
$c_1$ be the center of the minimal circumsphere of $\sigma$.
Let $r$ be the length of $\overline{p_0 c_0}$,
the radius of minimal circumsphere of $\tau$, 
and $\theta$ be the angle of $\angle c_0 p_1 c_1$
(Figure~\ref{fig:r_theta}).
From the fact \ref{enum:voronoi_1} of the Voronoi decomposition,
$\overline{c_0 c_1}$ is orthogonal to $\sigma$. Therefore 
we can represent $r_{\tau}$ and $r_{\sigma}$ as follows:
\begin{eqnarray*}
 r_{\tau} &=& r, \\
 r_{\sigma} &=& r \cos \theta.
\end{eqnarray*}

\begin{figure}[htbp]
 \centering
 \includegraphics[width=0.45\hsize]{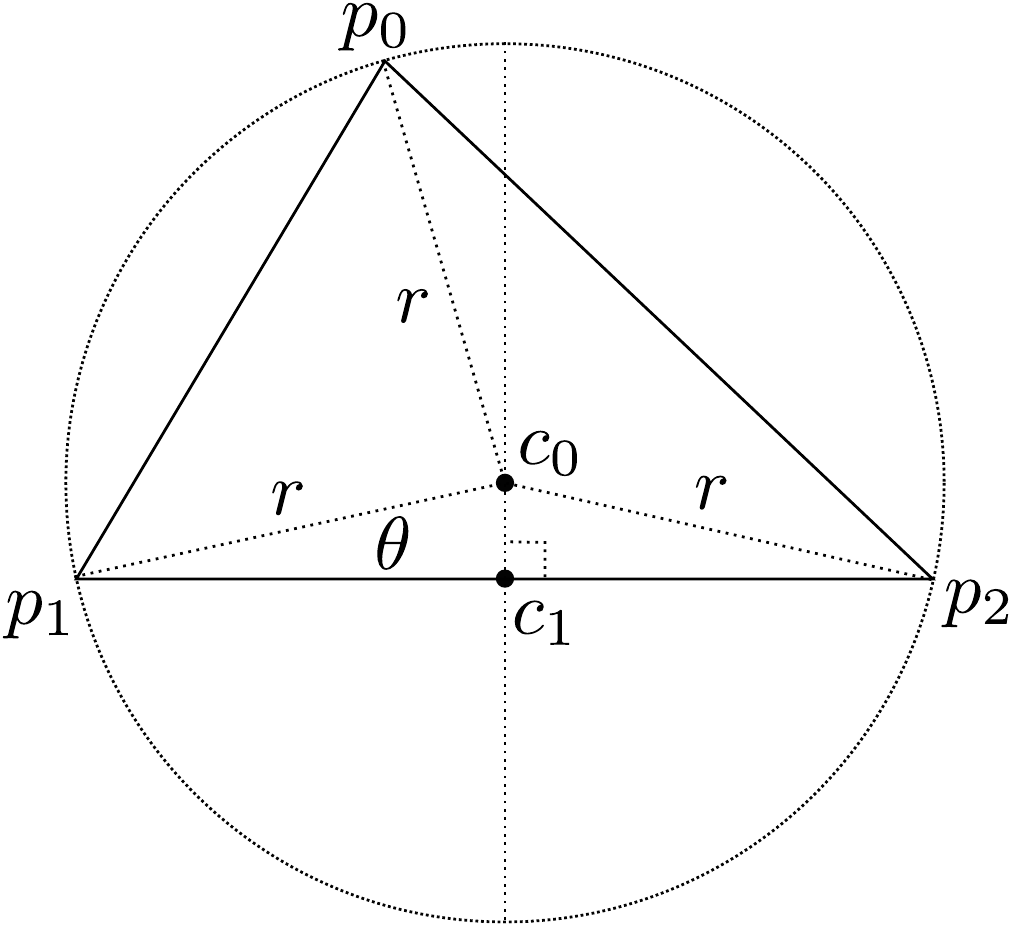}
 \ \ 
 \includegraphics[width=0.45\hsize]{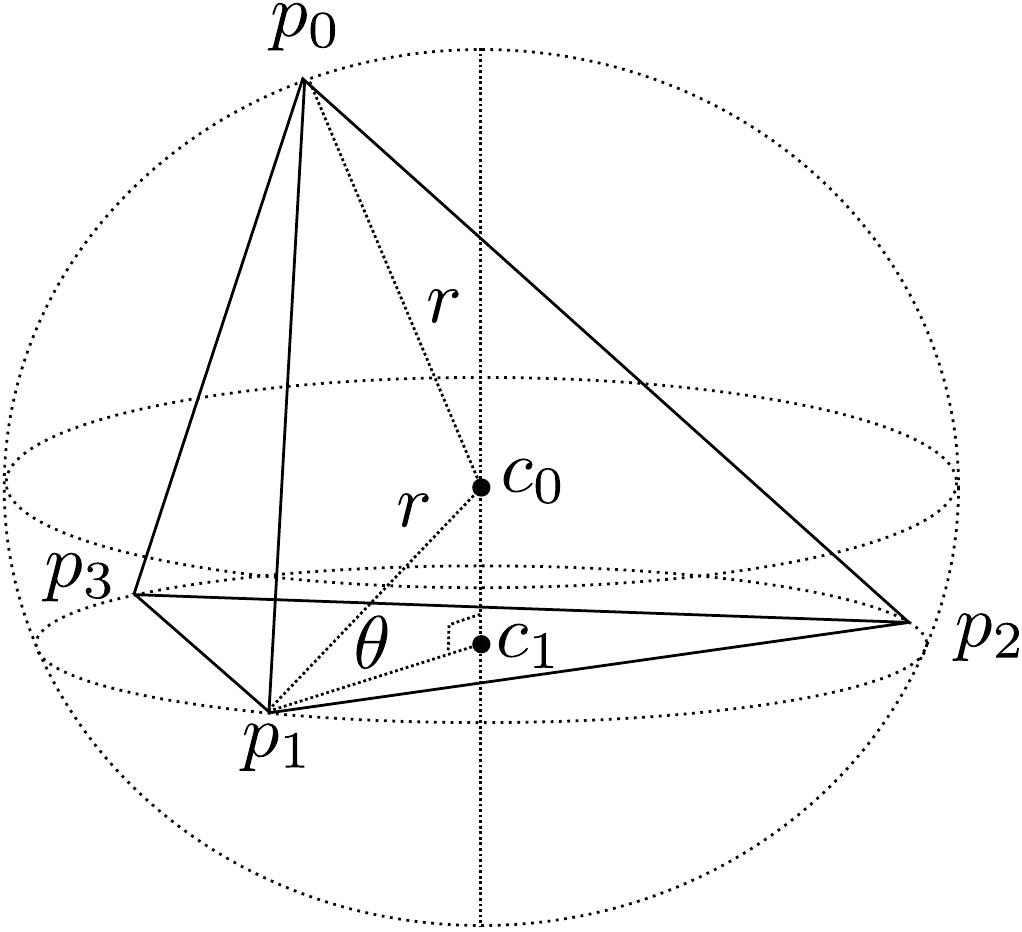}
 \caption{
 In the left figure, 
 $c_0$ is the center of the circumcircle of the 
 triangle $\{p_0, p_1, p_2\}$,
 $c_1$ is the center of $\overline{p_1 p_2}$,
 $r$ is the radius of the circumcircle,
 and $\theta$ is $\angle c_0 p_1 c_1$.
 In the right figure,
 $c_0$ is the center of the circumsphere of 
 the tetrahedron $\{p_0, p_1, p_2, p_3\}$,
 $c_1$ it the center of the circumcircle
 of the triangle $\{p_1, p_2, p_3\}$,
 $r$ is the radius of the circumsphere,
 and $\theta$ is $\angle c_0 p_1 c_1$.
 }
 \label{fig:r_theta}
\end{figure}

We consider a map $f:U \to \R^2$ given by
\begin{eqnarray*}
 f: u \mapsto 
  \begin{bmatrix}
   r_{\tau} \\ r_{\sigma}
  \end{bmatrix}
\end{eqnarray*}
assigning one birth-death pair.
This map is one component of the persistence map 
and can be decomposed as follows
\begin{equation*}
\xymatrix{
 u \ar@{|->}[r]^-{f_1}& (r, \theta) \ar@{|->}[r]^-{f_2} & (r_{\tau}, r_{\sigma})
  }.
\end{equation*}

We can easily show the following facts:
\begin{itemize}
\item $r_{\tau} = r_{\sigma}$ if and only if 
  $\theta = 0$.
\item  If $\theta = 0$, 
 \begin{equation*}
  Df_2 = 
   \begin{bmatrix}
    1 & 0 \\
    1 & 0
   \end{bmatrix}
 \end{equation*}
 holds.
\end{itemize}
From these facts, we have $\theta \approx 0$ for $r_{\tau} \approx r_{\sigma}$ and
 \begin{equation*}
  Df_2 \approx 
   \begin{bmatrix}
    1 & 0 \\
    1 & 0
   \end{bmatrix}.
 \end{equation*}
 Since the matrix of the right hand side is not surjective, $Df = Df_2 \circ Df_1$ has
 a singular value close to zero,
 and hence the derivative of the persistence map has 
 a singular value close to zero.

\subsubsection{For the case of Vietoris-Rips filtrations}
For the case of Vietoris-Rips filtrations, a birth-death pair $(b,d)$ with $b\approx d$ does not necessarily imply the existence of a singular value close to zero. 
%
%
We show such an example 
for a point cloud $P=\{A,B,C,D\}$ in $\R^2$ and $\pdiag_1(\vr(P))$.
We assume the followings:
\begin{itemize}
 \item $|\overline{AB}|, |\overline{BC}|, |\overline{CD}|
       < |\overline{AD}| < |\overline{BD}| < |\overline{AC}|$.
 \item $|\overline{AD}| \approx |\overline{BD}|$.
\end{itemize}
From the assumption, the triangle $ABD$ is close to an isosceles triangle.
The persistence diagram $\pdiag_1(\vr(P))$ has a unique
birth-death pair $(b, d)$ where $b = |\overline{AD}|$ and
$d = |\overline{BD}|$ (Figure~\ref{fig:rips_b_approx_d}).

\begin{figure}[htbp]
 \centering
 \includegraphics[width=0.7\hsize]{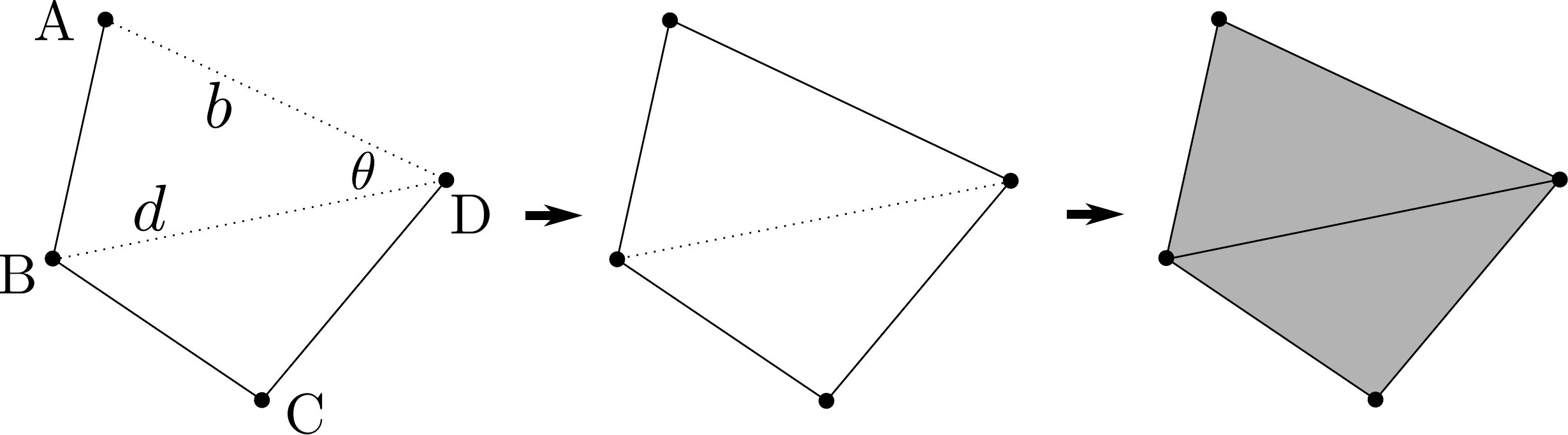}
 \caption{A point cloud with four points.
 $\pdiag_1(\vr(P))$
 has a unique birth-death pair $(b, d)$}
 \label{fig:rips_b_approx_d}
\end{figure}

The persistence map $f$ is
\begin{equation*}
 f = \begin{bmatrix}
      b \\
      d \\
     \end{bmatrix},
\end{equation*}
and from the computation in Lemma \ref{lemma:differentiable_g_rips},
$Df \cdot (Df)^T$ is
\begin{equation*}
 Df \cdot (Df)^T =
  \begin{bmatrix}
   2 & \cos \theta \\
   \cos \theta & 2 \\
  \end{bmatrix},
\end{equation*}
where $\theta = \angle ADB$.
Since eigenvalues of $Df \cdot (Df)^T$ are squares of singular values of
$Df$, the singular values are
$\sqrt{2 \pm \cos \theta}$
and are away from zero, although $b \approx d$.

\section{Computations}
\label{sec:computations}

In this section, we present some numerical examples of continuations of point clouds using
persistence diagrams. Alpha filtrations are used for all examples, and the coordinate system
described in Section~\ref{sec:symmetry} is adopted to eliminate the translation and rotation
symmetries.


\begin{ex}[Deformation of a tetrahedron]
\label{ex:thetrahedron}
As a first example, we consider alpha filtrations constructed from four points in $\R^3$
(a tetrahedron) and apply our continuation algorithm to $\pdiag_2$. We take as the initial
point cloud $P = \{u_0, u_1, u_2, u_3\}$, where $u_0 = (0, 0, 0)$, $u_1 = (8, 0, 0)$,
$u_2 = (5, 6, 0)$, and $u_3 = (4, 2, 6)$. The $2$nd persistence diagram of $P$ is
$\pdiag_2 = \{ (4.42719, 4.59015) \}$. From this initial data, we try to deform $\pdiag_2$ to
the target persistence data $\{ (8.42719, 8.89015) \}$
using our continuation method. Due to the coordinate system adopted to eliminate symmetries,
\eqref{eq:symmetry_restriction}, we have that the degree of freedom of the point cloud is six,
and the persistence map can be
expressed as $f \colon U \subset \R^6 \to \R^2$. For this example we used $\|\Delta v\| = 0.01$
as the step size in the continuation, and $\epsilon = 0$\footnote{The $2$-dimensional diagram
$\pdiag_2$ of a tetrahedron has at most one birth-death pair, hence we cannot cut off points
close to the diagonal and therefore we use $\epsilon = 0$.}.

Figures~\ref{fig:point_clouds_ex1} and \ref{fig:persistence_diagrams_ex1} show the point clouds
and the persistence diagrams during the continuation process.
In this computation we successfully reached the target persistence diagram $\{ (8.42719, 8.89015) \}$.
However notice that there seems to be a non-smooth point on the blue curves given by the point
clouds during continuation in Figure~\ref{fig:point_clouds_ex1}. To try to understand this event, we look at
the birth radii of the $2$-simplices during the continuation process (Figure~\ref{fig:birth_radii_ex1}).
Notice that at some point during the continuation process, two of the birth radii coincide, hence breaking
the second condition of the alpha general position assumption (Definition~\ref{definition:alpha_general}).
This point corresponds exactly to the non-smooth point in Figure~\ref{fig:birth_radii_ex1}. This is a point
where the derivative $Df$ is not uniquely defined, hence the continuation curve is not smooth there. The
fact that we are performing the continuation numerically, and hence the birth radii are very close but
not exactly equal, makes it possible to go forward with the continuation process. At each step of the
continuation the birth radius that is slightly larger is used to compute $Df$, and the $2$-simplex
corresponding to this radius can change at each continuation step. We can see this from the
singular values of $Df$ show in Figure~\ref{fig:singular_values_ex1}.

\begin{figure}[htbp!]
\centering
\includegraphics[width=0.45\textwidth]{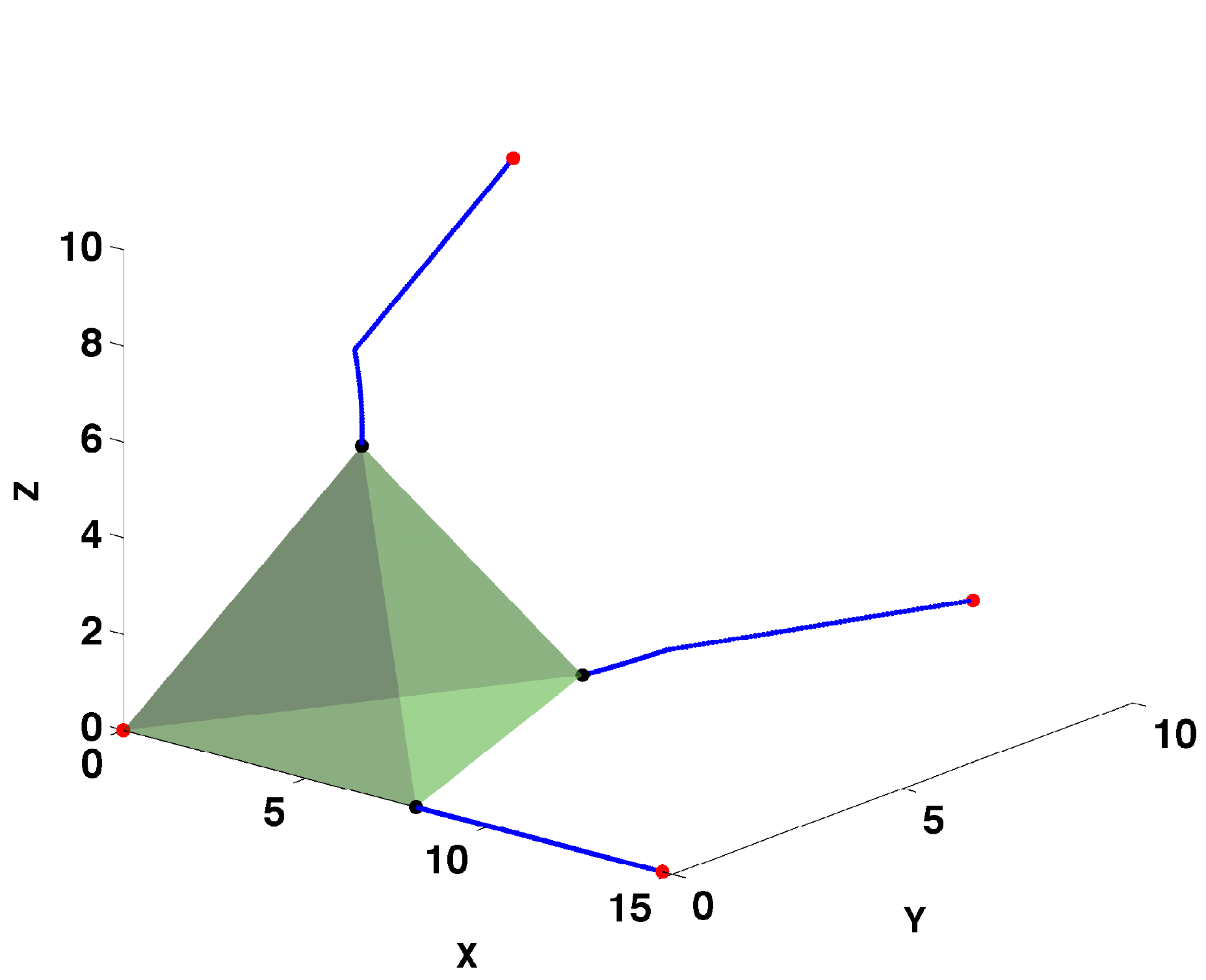}
\caption{Deformation of the point cloud during continuation. The black dots are the initial
point cloud and the red dots are the final one. The blue curves show the movement
of the points during the continuation. Note that the first vertex is fixed at the origin and
the second one is only allowed to move along the $x$-axis.}
\label{fig:point_clouds_ex1}
\end{figure}

\begin{figure}[htbp!]
\centering
\includegraphics[width=0.45\textwidth]{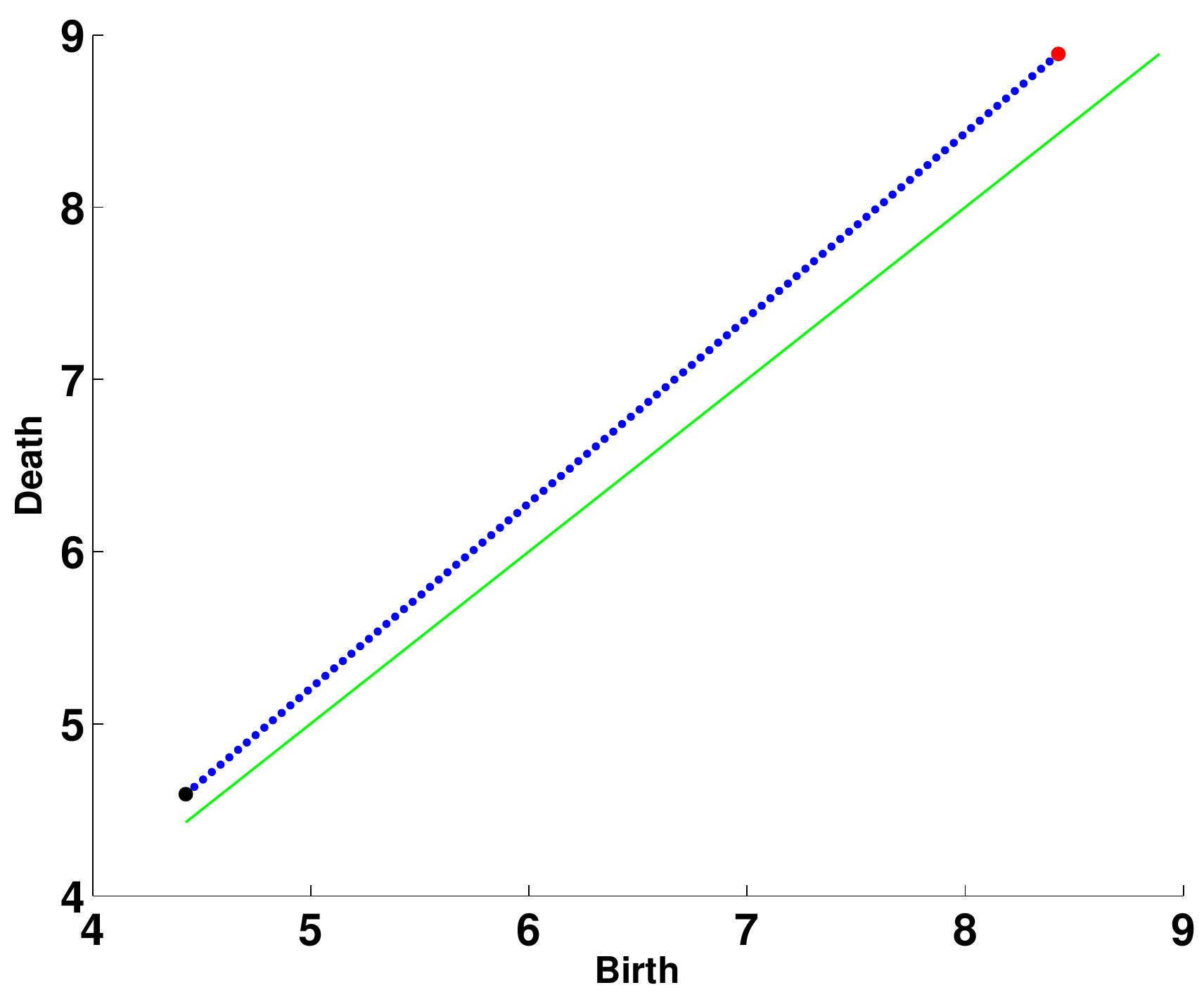}
\caption{Persistence diagrams $\pdiag_2$ during the continuation process. The black dot
represents the initial persistence diagram and the red dot represents the target diagram.
The blue dots represent the persistence diagrams during the continuation process.}
\label{fig:persistence_diagrams_ex1}
\end{figure}

\begin{figure}[htbp!]
\centering
\includegraphics[width=0.45\textwidth]{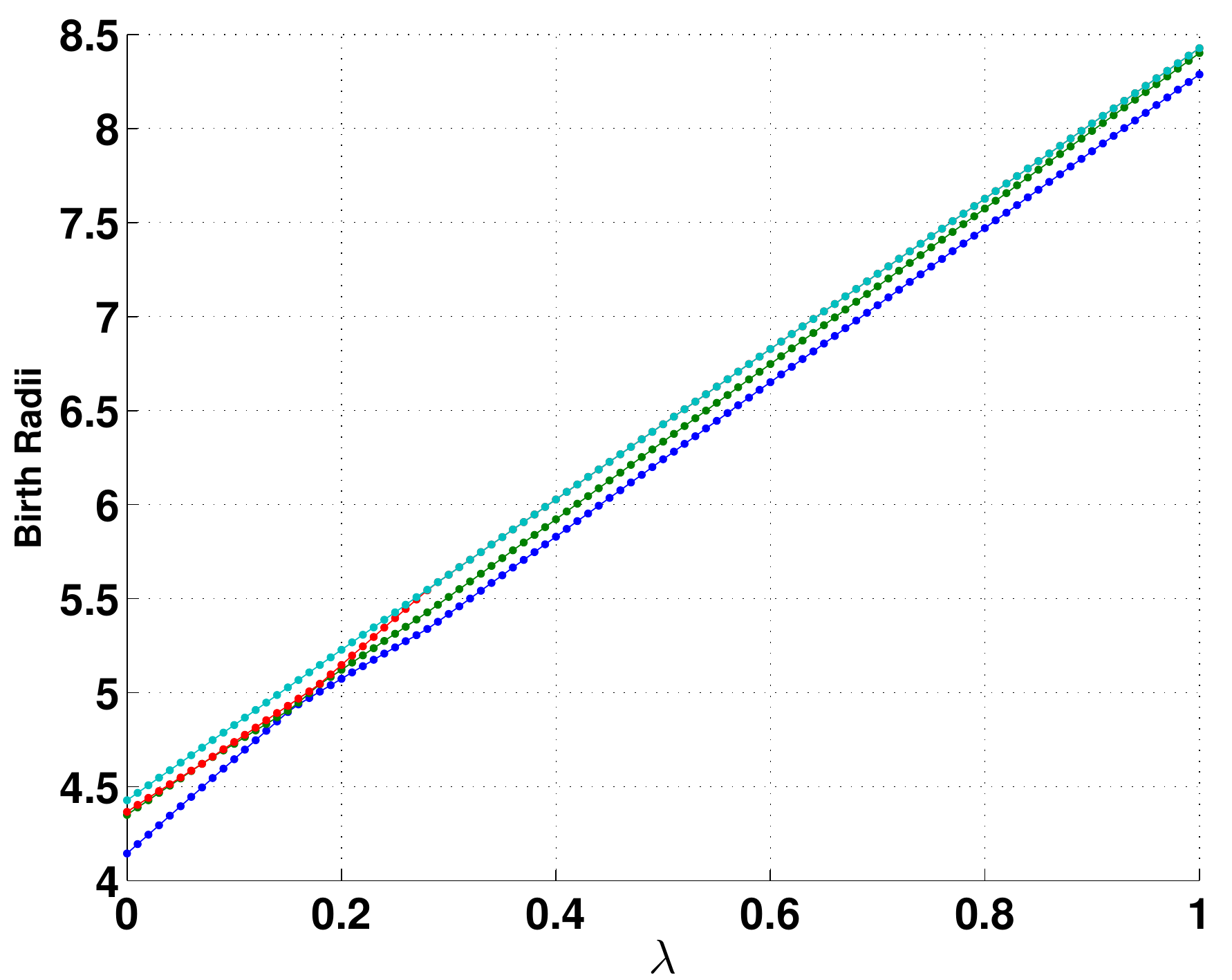}
\caption{Birth radii of the $2$-simplices along the continuation curves in
Figure~\ref{fig:point_clouds_ex1}.}
\label{fig:birth_radii_ex1}
\end{figure}

\begin{figure}[htbp!]
\centering
\includegraphics[width=0.45\textwidth]{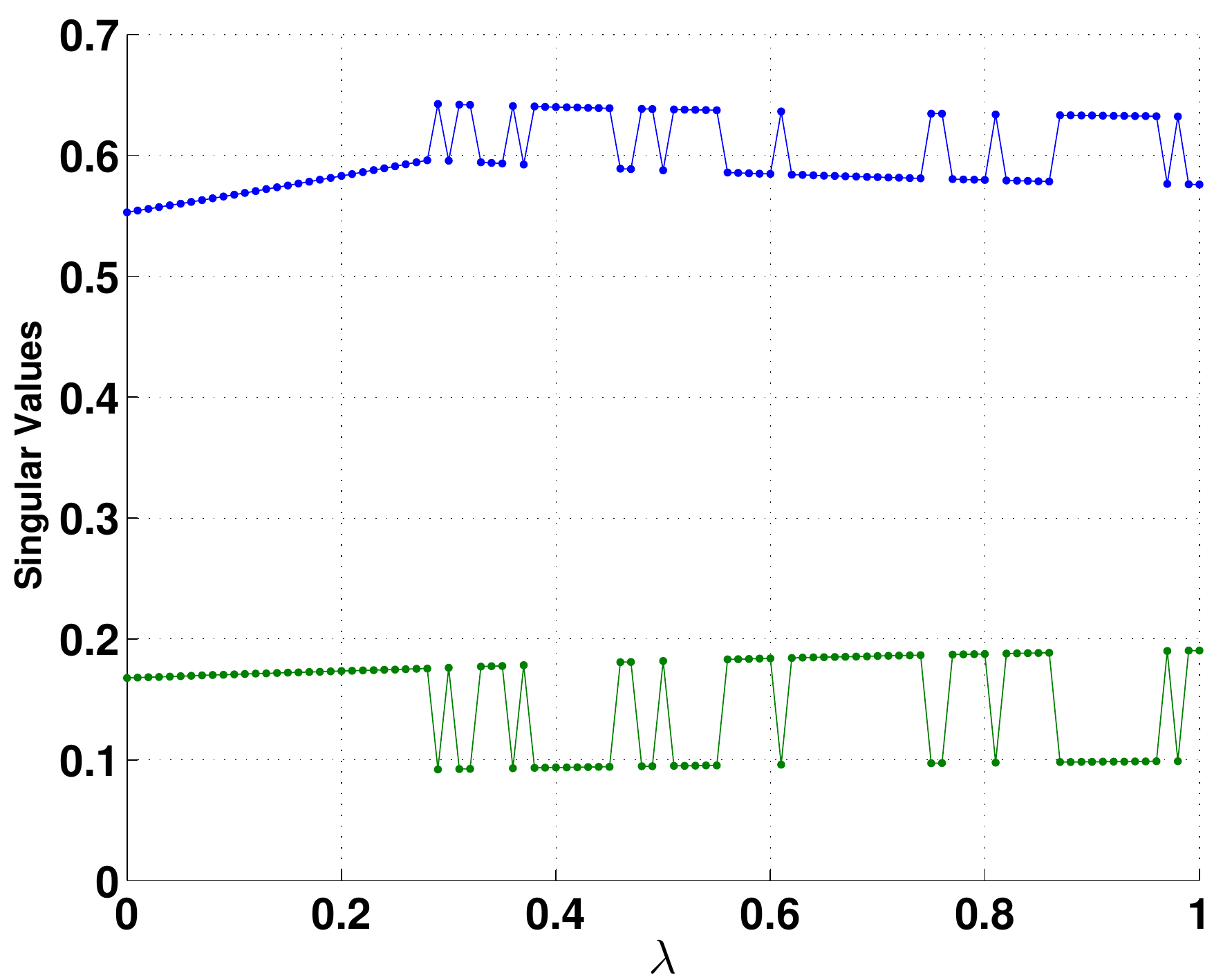}
\caption{Singular values of the derivative $Df$ along the continuation in Figure~\ref{fig:point_clouds_ex1}.}
\label{fig:singular_values_ex1}
\end{figure}

\end{ex}


\begin{ex}[Image of the persistence map]
\label{ex:existence_regions}
In this example we use the same point cloud (tetrahedron) as in Example~\ref{ex:thetrahedron}
to explore the image of the persistence map. For a tetrahedron, the image of the persistence map is the strip
region between the diagonal and the line of persistence diagrams of regular tetrahedrons
(see Figure~\ref{fig:persistence_diagrams_regions2} and Theorem~\ref{them:ratio_bound_triangle} in the
Appendix for a proof of the $\pdiag_1$ case). In this example we used $\|\Delta v\| = 0.001$ and $\epsilon = 0$.

From the initial persistence diagram $\pdiag_2 = \{ (4.42719, 4.59015) \}$ we try to continue to
the target persistence $\{ (6.42719, 7.09015) \}$, which is outside of the image of the persistence map.
Hence, as expected, in this case we fail to reach the target persistence and can only continue up to the
boundary of the image, as we can see in Figure~\ref{fig:persistence_diagrams_regions2}. 
As we approach the boundary of the image, the method fails because the number of Newton iterations needed
for convergence increases dramatically as is shown in Figure~\ref{fig:newton_iterations2}.
In the last steps of the continuation the birth radii of the $1$-simplices are very similar and the
birth radii of the $2$-simplices are all virtually the same as we can see in Figures~\ref{fig:birth_radii_dim1}
and~\ref{fig:birth_radii_dim2}, hence confirming that we have continued to a regular tetrahedron.
Notice also from Figures~\ref{fig:birth_radii_dim1} and~\ref{fig:birth_radii_dim2} that two or more
birth radii are equal during the continuation, hence the general position assumption
(Definition~\ref{definition:alpha_general}) is violated. However, as noted in
Example~\ref{ex:thetrahedron}, the continuation method still works as long as we are within
the image of the persistence map.

\begin{figure}[htbp!]
\centering
\includegraphics[width=0.5\textwidth]{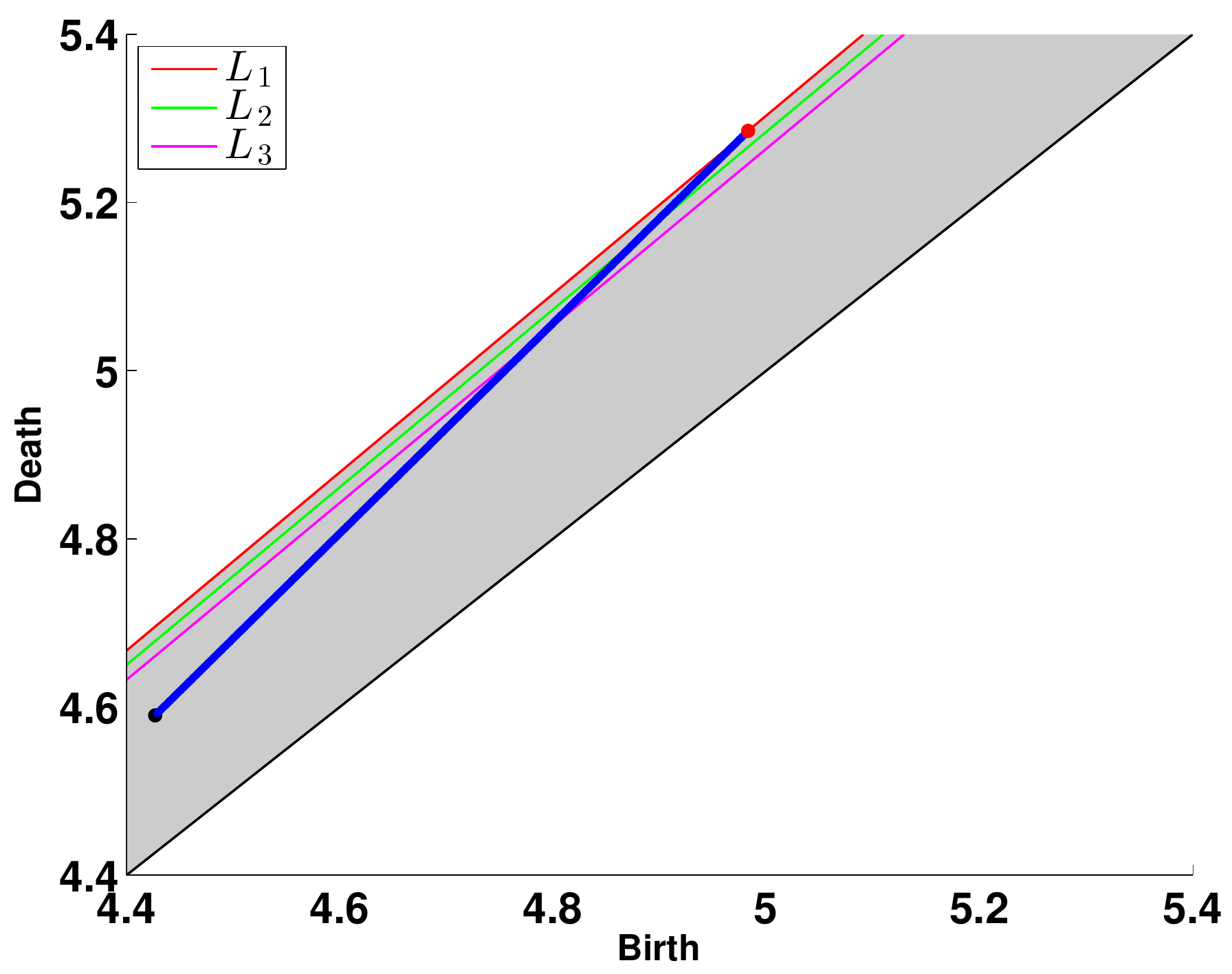}
\caption{The image of the persistence map of $\pdiag_2$ for
one tetrahedron (four points in $\R^3$) is the shaded region between the diagonal and the line
$L_1$ of the persistence diagrams of regular tetrahedrons (obtained by similarity deformations).
The lines $L_2$ and $L_3$ correspond to the persistence diagrams of tetrahedrons with three
and two congruent faces, respectively. The blue line shows the continuation of the initial point
cloud all the way to the boundary of the image.}
\label{fig:persistence_diagrams_regions2}
\end{figure}

\begin{figure}[htbp!]
\centering
\includegraphics[width=0.45\textwidth]{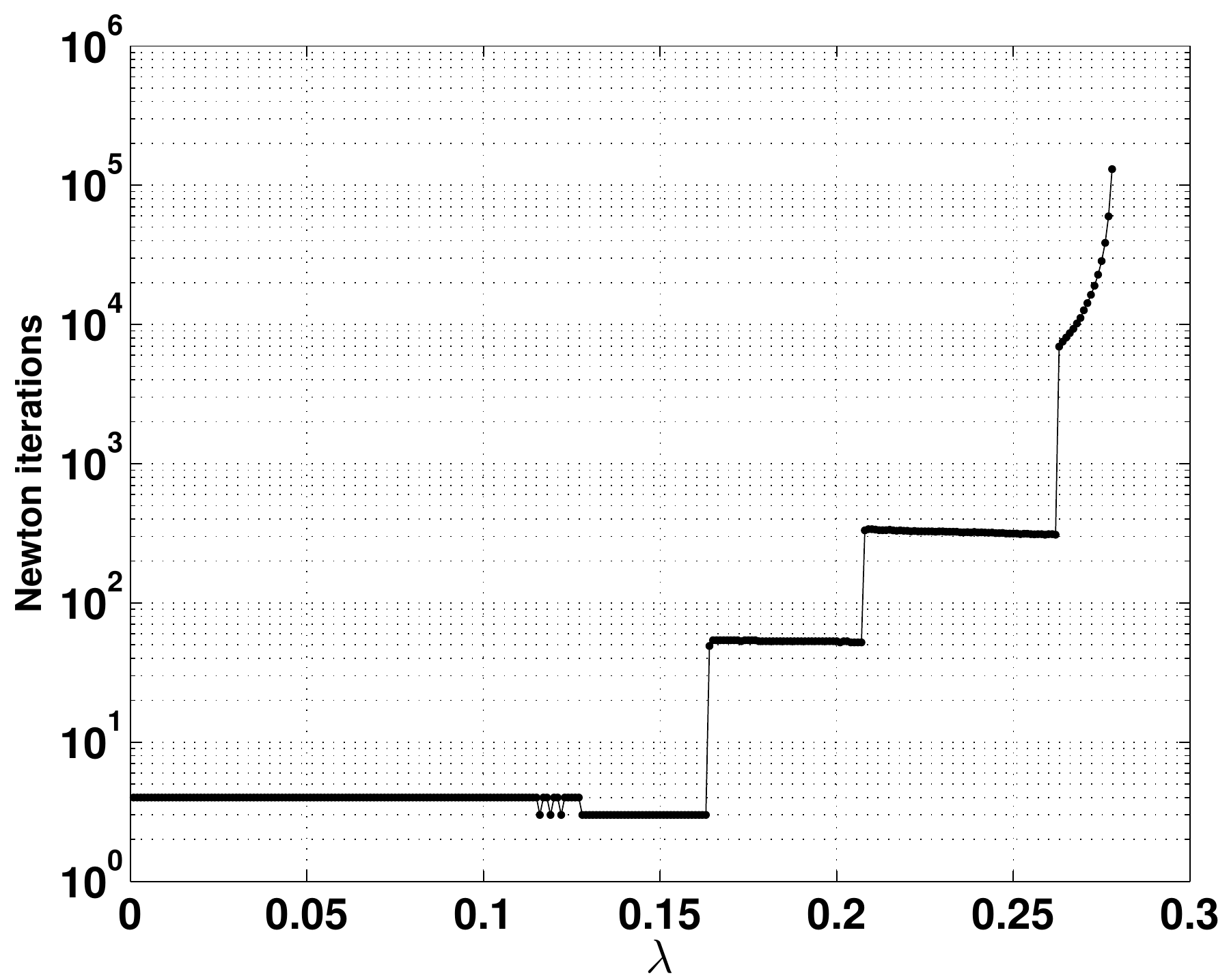}
\caption{Number Newton iterations during each step of the continuation in
Example~\ref{ex:existence_regions}.}
\label{fig:newton_iterations2}
\end{figure}

\begin{figure}[htbp!]
\centering
\includegraphics[width=0.45\textwidth]{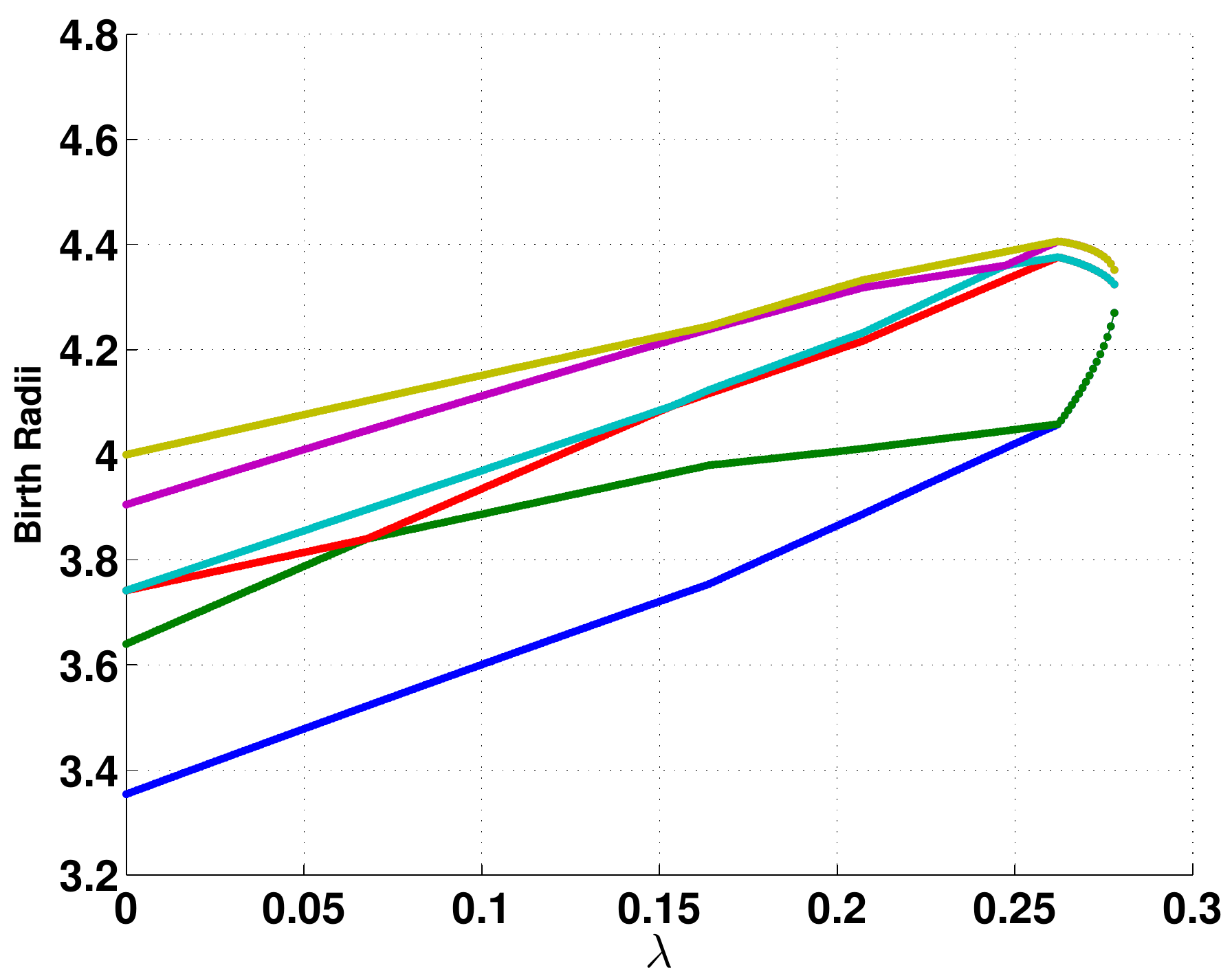}
\caption{Birth radii of the $1$-simplices along the continuation in
Example~\ref{ex:existence_regions}.}
\label{fig:birth_radii_dim1}
\end{figure}

\begin{figure}[htbp!]
\centering
\includegraphics[width=0.45\textwidth]{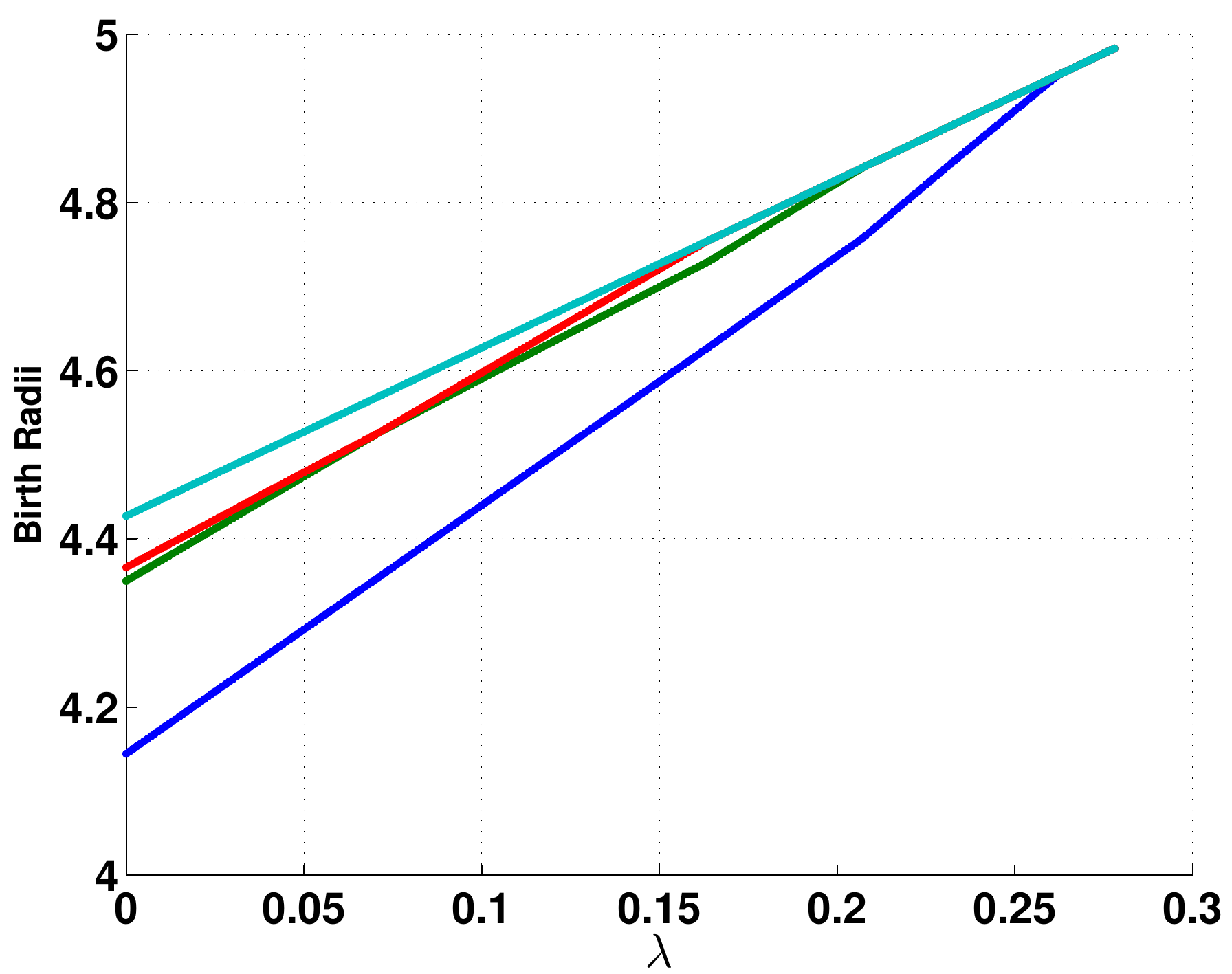}
\caption{Birth radii of the $2$-simplices along the continuation in
Example~\ref{ex:existence_regions}.}
\label{fig:birth_radii_dim2}
\end{figure}

\end{ex}


\begin{ex}[Towards the diagonal]
\label{ex:towards_diagonal}
In this example we take the point cloud $P = \{u_0, u_1, u_2, u_3\}$, where $u_0 = (0, 0, 0)$,
$u_1 = (9.991, 0, 0)$, $u_2 = (4.9955, 8.65246, 0)$, and $u_3 = (4.9955, 2.88415, 8.15762)$,
which represents a nearly regular tetrahedron, and try to continue the persistence diagram
$\pdiag_2$ towards the diagonal. From the discussion in Section~\ref{sec:zsv} this will lead
to singular values close to zero. The $2$nd persistence diagram of $P$ is
$\pdiag_2 = \{ (5.76831, 6.11821) \}$, and the target persistence diagram is set to be $\{ (5.94841, 5.94841) \}$
which is on the diagonal. In the computations we used $\|\Delta v\| = 0.001$ and $\epsilon = 0$.

In this case the computations work well all the way to a point nearly on the diagonal. In
Figure~\ref{fig:persistence_diagrams5} we show the persistence diagrams along the continuation.
The singular values of the derivative are shown in Figure~\ref{fig:singular_values5}.
As expected, one of the singular values approaches zero towards the end of the
continuation. However, in spit of this, the continuation works all the way to a point
essentially on the diagonal. Note that we cannot continue to a point exactly on the
diagonal, since the persistence diagram would be empty in that case. However the
persistence diagram that we arrive at the end of the continuation in this example is
only ``numerically'' on the diagonal, that is, it is on the diagonal up to the error tolerance
of the Newton-Raphson method.
As described in Section~\ref{sec:newton_pseudo_inverse} the method will fail if we try
to continue to a point exactly on the diagonal, since in that case we would have a zero
singular value.

\begin{figure}[htbp!]
\centering
\includegraphics[width=0.45\textwidth]{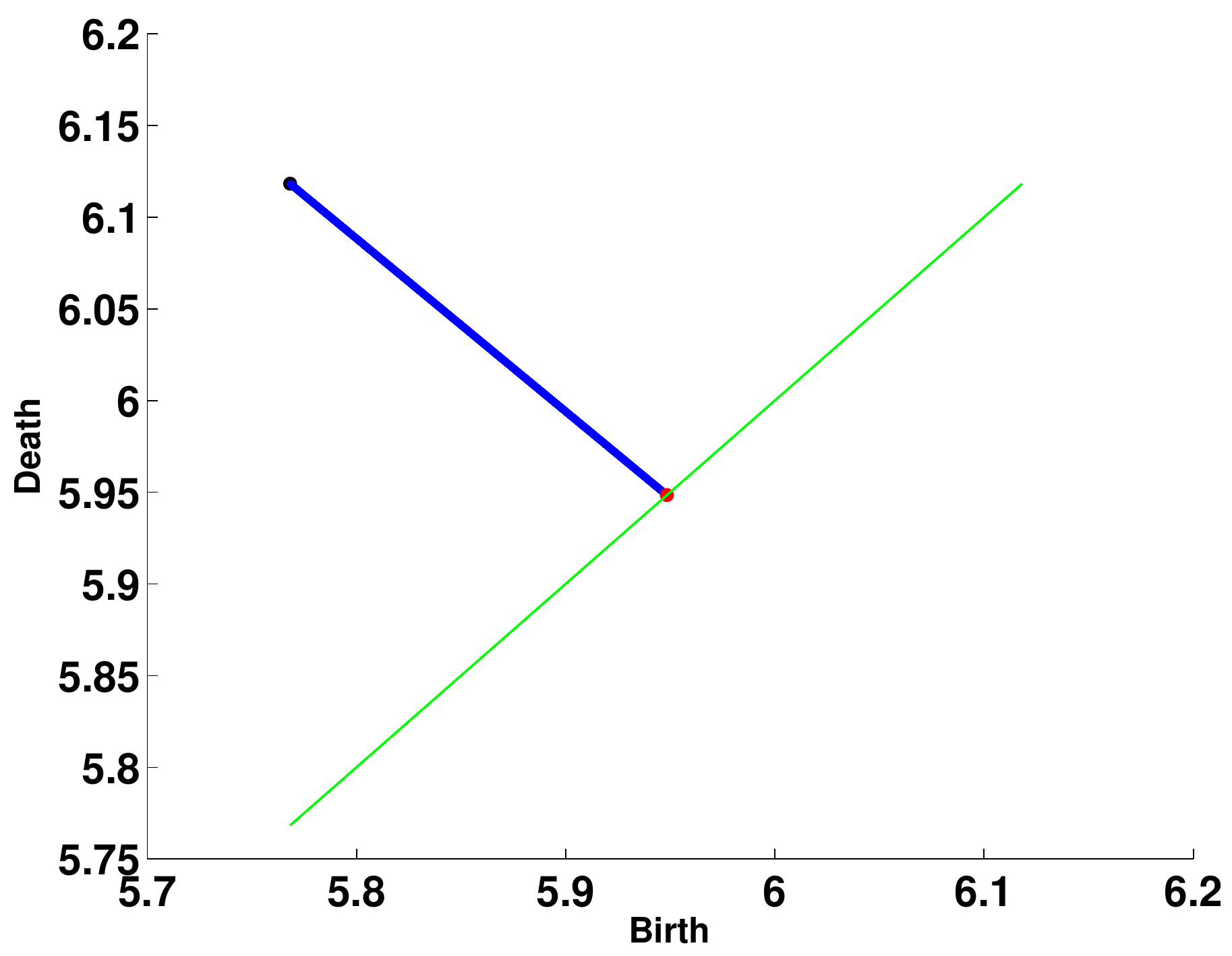}
\caption{Persistence diagrams $\pdiag_2$ during the continuation process. The black dot
represents the initial persistence diagram and the red dot represents the final diagram.
Notice that the final diagram is ``numerically'' on the diagonal.}
\label{fig:persistence_diagrams5}
\end{figure}

\begin{figure}[htbp!]
\centering
\includegraphics[width=0.45\textwidth]{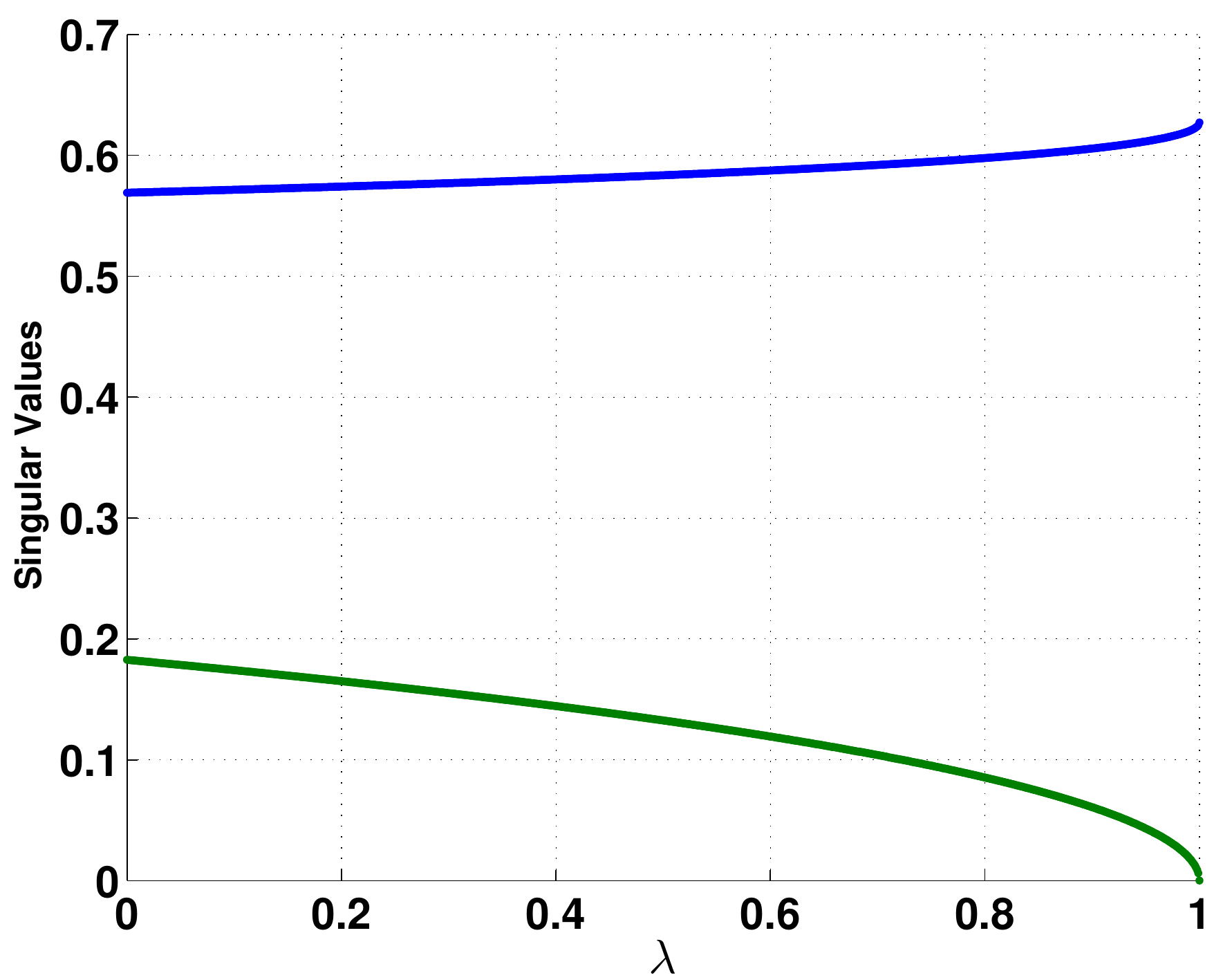}
\caption{Singular values of the derivative $Df$ along the continuation in Example~\ref{ex:towards_diagonal}. Notice that one of the singular
values approach zero as we reach the end of the continuation.}
\label{fig:singular_values5}
\end{figure}


\end{ex}


\begin{ex}[Continuation of $\pdiag_1$]
\label{ex:pd1_example}
In this example we take the point cloud data $P = \{u_0, u_1, u_2, u_3\}$, where $u_0 = (0, 0, 0)$,
$u_1 = (1, 0, 0)$, $u_2 = (1.1, 1.2, 0)$, and $u_3 = (0.5, 0.6, 1.3)$, and try to continue 
the initial persistence diagram $\pdiag_1 = \{ (0.758288, 0.803195), (0.776209, 0.834393) \}$
to the target $1$-dimensional persistence diagram $\{ (0.770801, 0.817236), (0.798346, 0.863075) \}$.
In these computations we used $\|\Delta v\| = 0.001$ and $\epsilon = 0$. There were only two points in
the persistence diagram during all the steps of the continuation, hence the choice $\epsilon = 0$.
The computations worked well all the way to the target persistence. In Figure~\ref{fig:persistence_diagrams8}
we show the point cloud and the persistence diagrams along the continuation.


\begin{figure}[htbp!]
\centering
\includegraphics[width=0.45\textwidth]{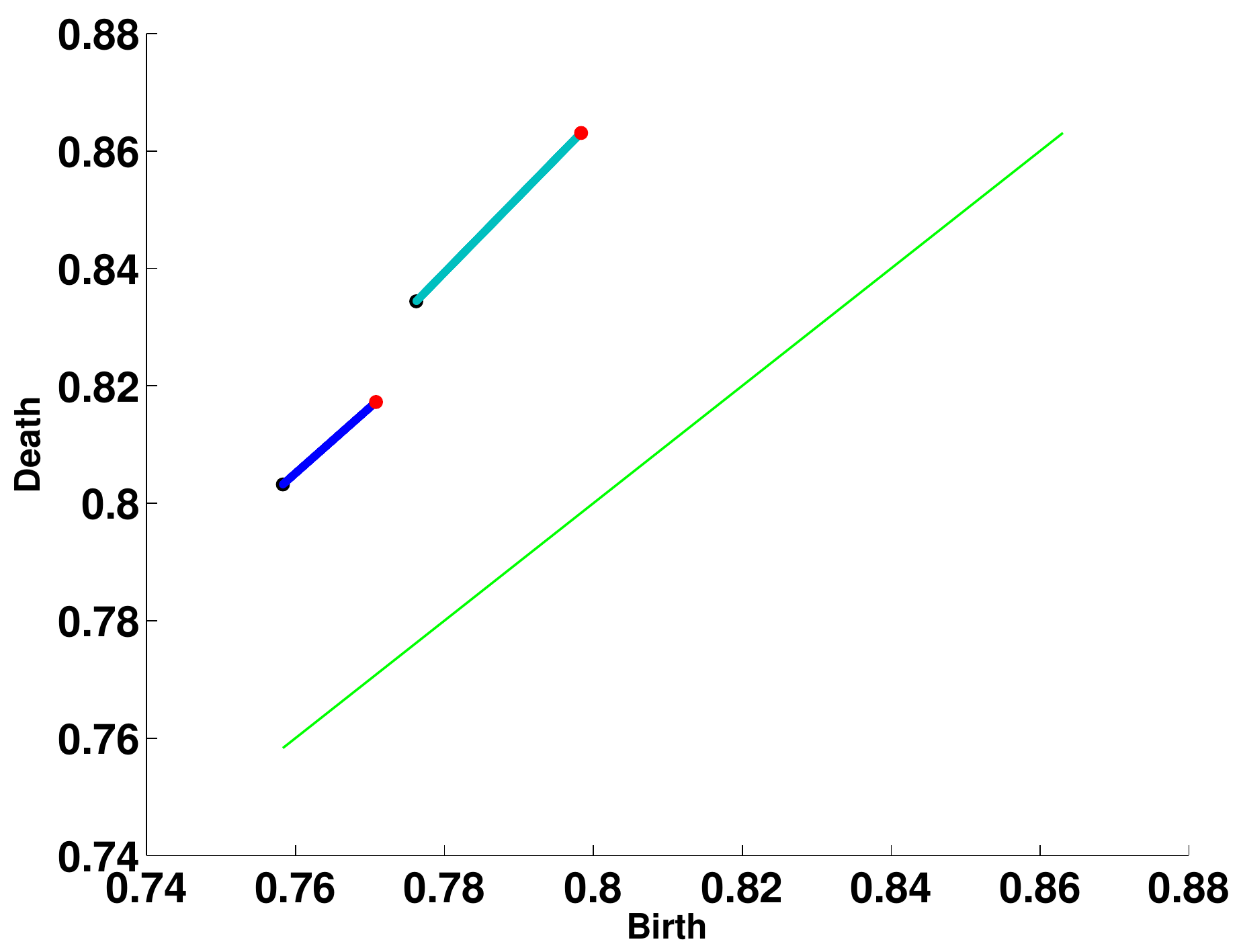}
\caption{Persistence diagrams $\pdiag_1$ during the continuation process. The black dots
represent the initial persistence diagram and the red dots represent the final diagram.}
\label{fig:persistence_diagrams8}
\end{figure}

\end{ex}


\begin{ex}[Deformation of a dodecahedron]
\label{ex:dodecahedron_example}
In this example we take as the initial point cloud the vertices of a regular dodecahedron and
try to apply our continuation method 
to increase both the birth and the death radii of the $\pdiag_2$ generator.
In these computations we used $\|\Delta v\| = 0.01$ and $\epsilon = 10^{-3}$.
The continuation works all the way to the target persistence diagram.
Figure~\ref{fig:point_clouds9} shows the deformation of
the point cloud and Figure~\ref{fig:persistence_diagrams9} shows the
diagrams along the continuation.

\begin{figure}[htbp!]
\centering
\includegraphics[width=0.5\textwidth]{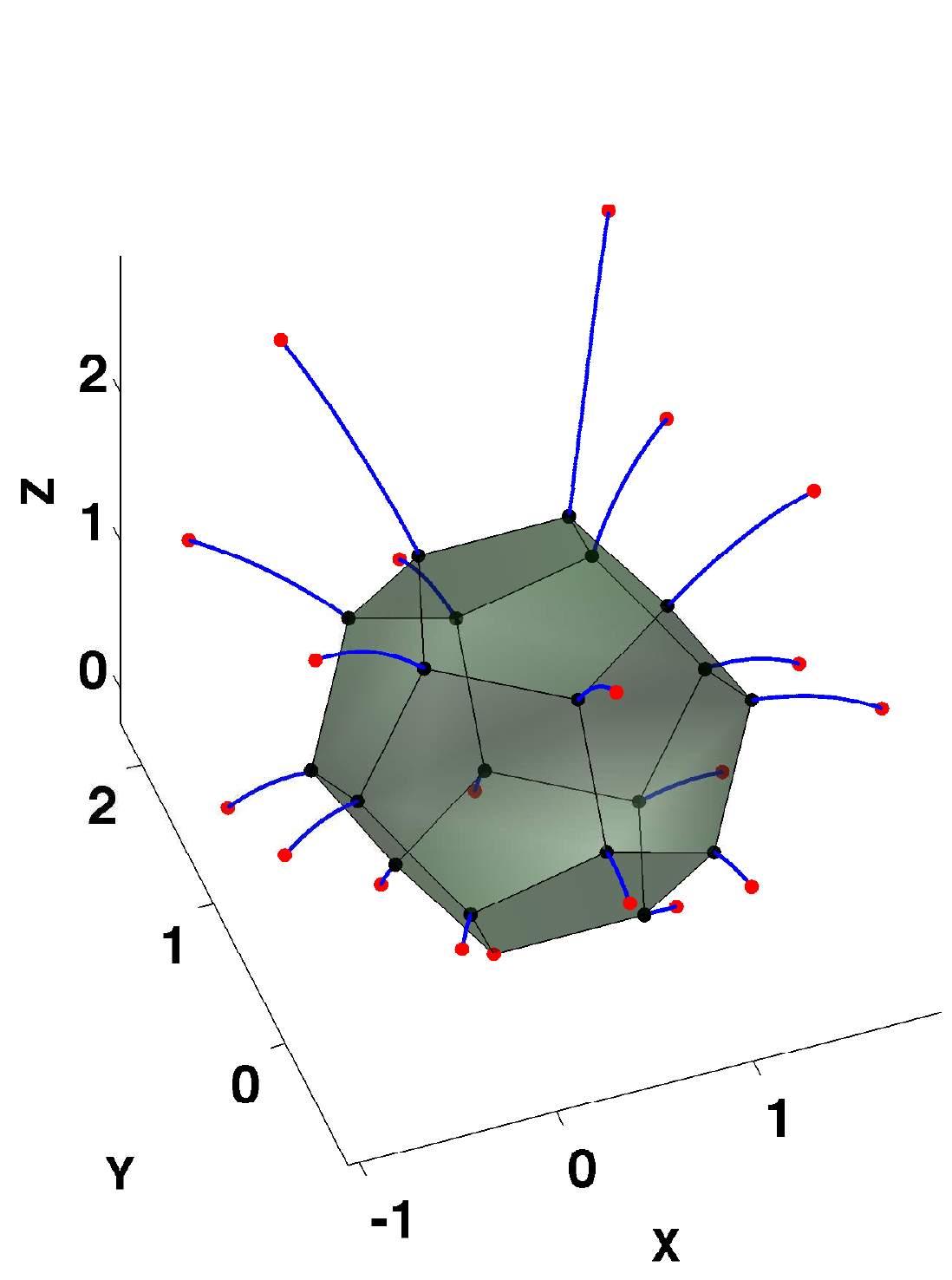}
\caption{Deformation of the point cloud during continuation. The black dots are the initial
point cloud and the red dots are the final one. The blue curves show the movement
of the points during the continuation.}
\label{fig:point_clouds9}
\end{figure}

\begin{figure}[htbp!]
\centering
\includegraphics[width=0.45\textwidth]{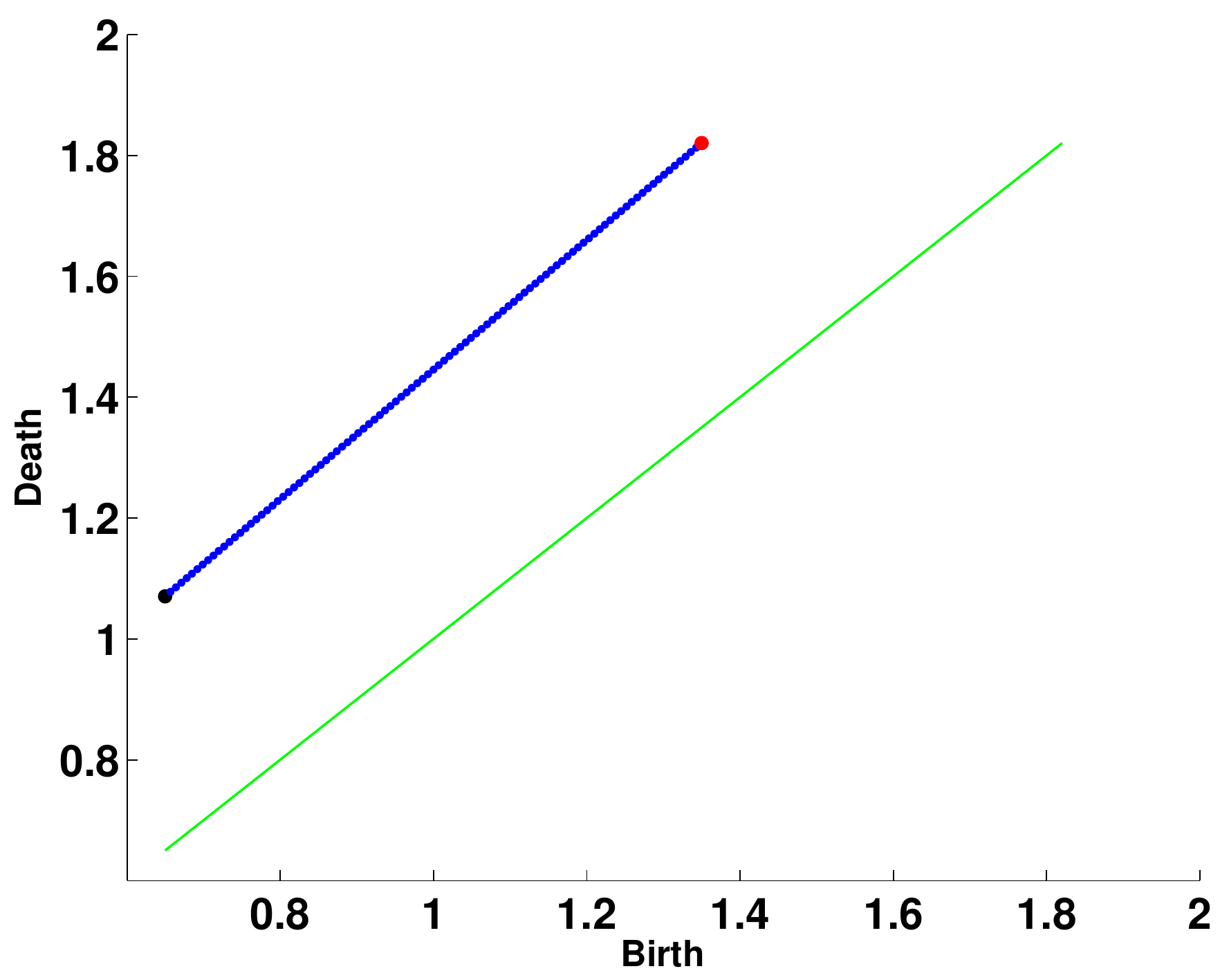}
\caption{Persistence diagrams $\pdiag_2$ during the continuation process. The black dot
represents the initial persistence diagram and the red dot represents the final diagram.}
\label{fig:persistence_diagrams9}
\end{figure}

\end{ex}


\begin{ex}[Deformation of a sphere]
\label{ex:sphere_example}
In this example we take as the initial point cloud $100$ uniform point on a sphere and
try to apply our continuation method to the largest generator of
$\pdiag_2$. In these computations we used $\|\Delta v\| = 0.03$ and $\epsilon = 10^{-5}$.
The continuation works all the way to the target persistence diagram. 
Figure~\ref{fig:point_clouds14}
shows the deformation of the point cloud and Figure~\ref{fig:persistence_diagrams14}
shows the diagrams along the continuation.

\begin{figure}[htbp!]
\centering
\includegraphics[width=0.5\textwidth]{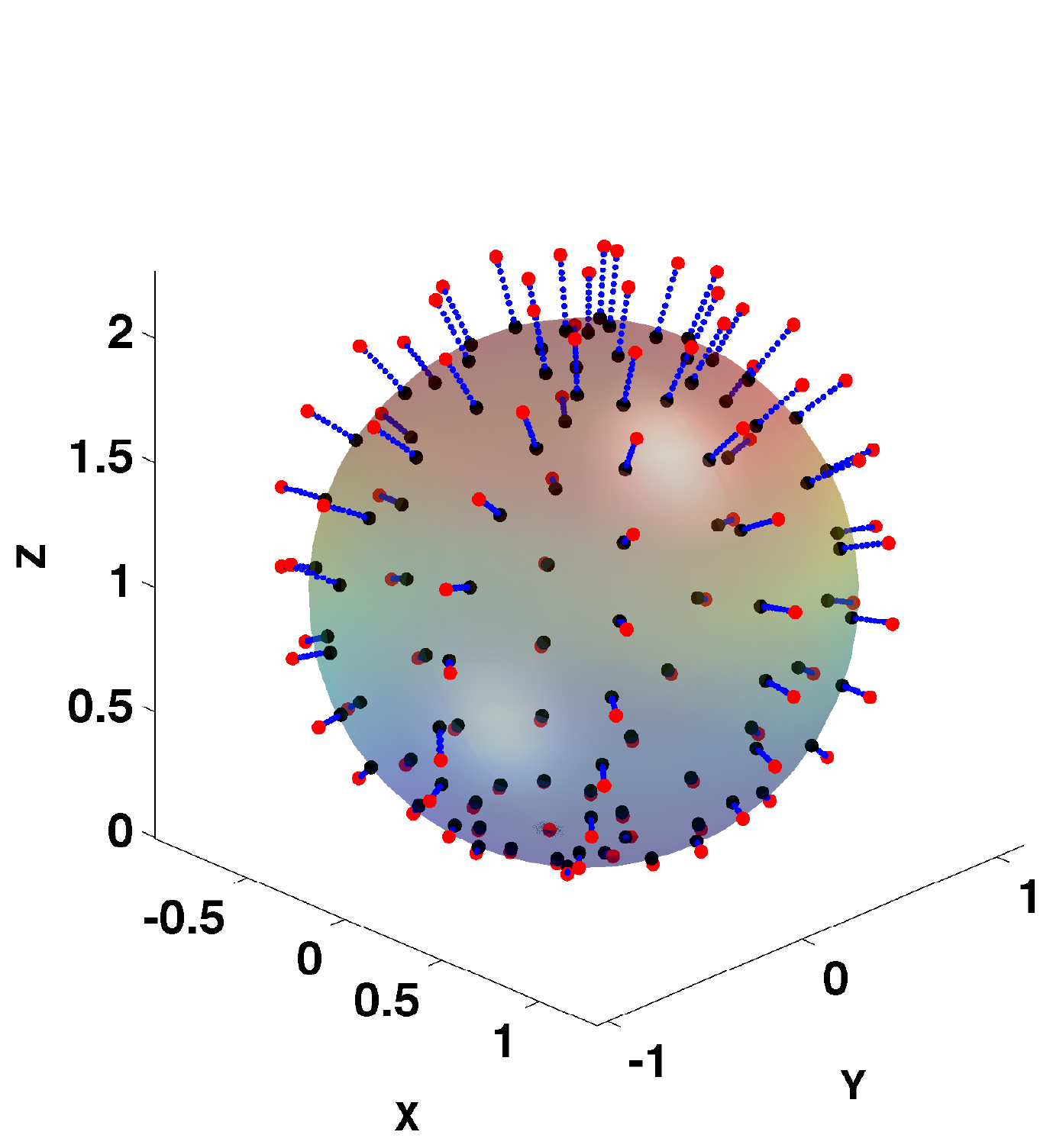}
\caption{Deformation of the point cloud during continuation. The black dots are the initial
point cloud and the red dots are the final one. The blue dots show the movement
of the points during the continuation.}
\label{fig:point_clouds14}
\end{figure}

\begin{figure}[htbp!]
\centering
\includegraphics[width=0.45\textwidth]{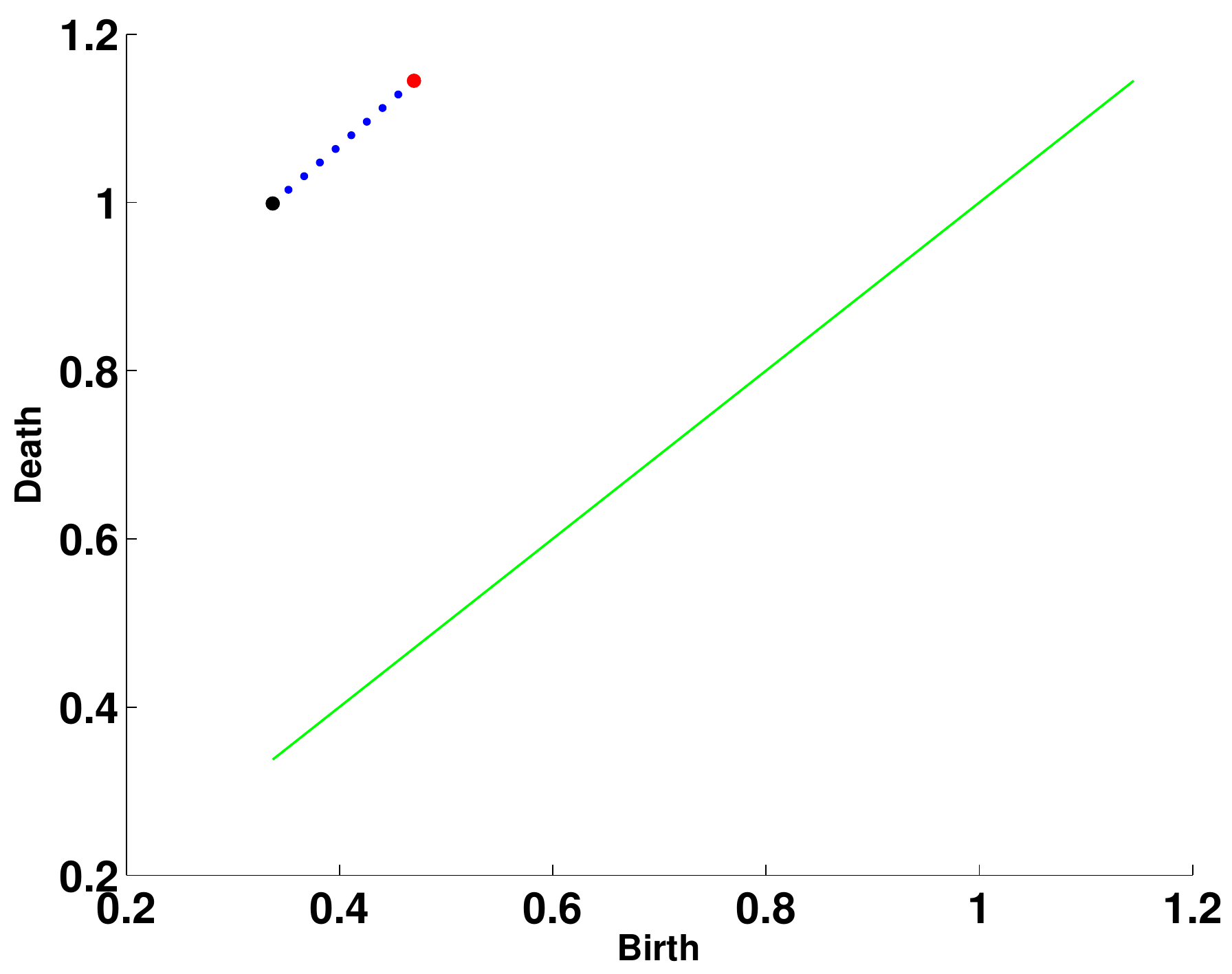}
\caption{Persistence diagrams $\pdiag_2$ during the continuation process. The black dot
represents the initial persistence diagram and the red dot represents the final diagram.}
\label{fig:persistence_diagrams14}
\end{figure}

\end{ex}


\section{Conclusion}
\label{sec:conclusion}

In this paper, we have studied the continuation of point clouds by persistence diagrams. In the following,
we list some future improvement of our method.

\begin{enumerate}
\setlength{\parskip}{0cm}
\setlength{\itemsep}{0cm}
\item In the presented method, we have treated persistence diagrams in an Euclidean space whose dimension
is determined by the input persistence diagram. This vectorization is simple and describes the essential part
for the continuation method. However, it does not allow to change the number of generators in the persistence
diagrams during the continuation because of the fixed dimension of the Euclidean space. To overcome this
restriction, it can be useful to vectorize the persistence diagrams into a bigger space and construct a similar
continuation method. The space of persistence landscape \cite{landscape} or a vectorization using kernel
methods \cite{kernel} should be considered as possible candidates.
\item Our algorithm for computing persistence diagrams in this paper is not sophisticated,  and hence, there is room
for improvement. Standard reduction methods such as \cite{morse_reduction} can be implemented and will
reduce the computational cost. Furthermore, since the changes in the point clouds at each step of the
continuation is supposed to be  small, the vineyard algorithm \cite{vineyard} can effectively work for fast
computations. 
\end{enumerate}

\section*{Acknowledgments} 
The authors would like to thank Shouhei Honda for useful discussions.
M. G. was partially supported by FAPESP grants 2013/07460-7 and 2010/00875-9,
and by CNPq grant 306453/2009-6, Brazil. Y. H. and I. O. were partially supported
by JSPS Grant-in-Aid (24684007, 26610042).

\section*{Appendix. Image of the Persistence Map of $\pdiag_1$ for a triangle}

\begin{thm}
\label{them:ratio_bound_triangle}
  If a point cloud $P$ has only three points,
  $\pdiag_1(\alp(P))$ has at most one birth-death pair. 
  If $\pdiag_1(\alp(P))$ has such a pair $(b, d)$,
  the ratio $d/b$ is smaller than or equal to $2/\sqrt{3}$.
  Moreover,
  $d/b = 2/\sqrt{3}$ if and only if
  the triangle is regular.
\end{thm}

\begin{proof}
  From the basic properties of alpha filtrations, 
  $\pdiag_1(\alp(P)) = \{(b,d)\}$ if and only if the triangle is
  acute. If not, $\pdiag_1(\alp(P))$ is empty.  
  Hence we assume that the triangle is acute.
  
  Let $p_0, p_1, p_2$ be the three vertices of the triangle, $c$ be the
  center of the circumcircle, $r$ be the radius of the circumcircle, and
  $\theta_0, \theta_1, \theta_2$ be $\angle c p_1 p_2$,
  $\angle c p_2 p_0$, and $\angle c p_0 p_1$, respectively 
  (Figure~\ref{fig:triangle_birthdeath}).
 
  \begin{figure}[htbp]
    \centering
    \includegraphics[width=0.2\textwidth]{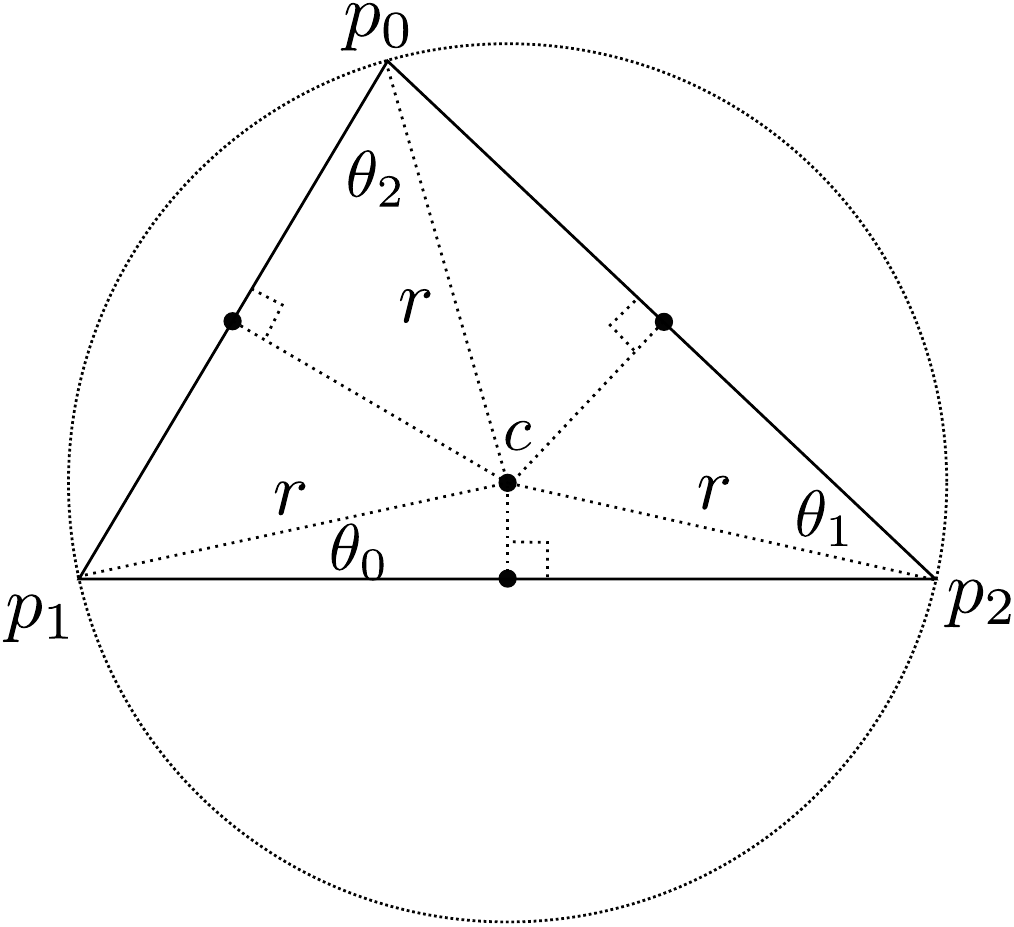}
    \caption{The relation of the circumcircle and internal angles of the
      triangle $\{p_0, p_1, p_2\}$}
    \label{fig:triangle_birthdeath}
  \end{figure}
  
  Hence, $|p_1 p_2| = 2r \cos\theta_0$, $|p_2 p_0| = 2r \cos\theta_1$,
  $|p_0 p_1| = 2r \cos\theta_2$ and the birth time $b$ is
  \begin{align*}
    b =\max\{ r \cos\theta_0, r \cos\theta_1, r \cos\theta_2\},
  \end{align*}
  and the death time $d$ is $r$.
  The ratio of $d/b$ is
  \begin{align*}
    &r / \max\{ r \cos\theta_0, r \cos\theta_1, r \cos\theta_2\} \\
    =& (\max\{ \cos\theta_0, \cos\theta_1, \cos\theta_2\})^{-1}.
  \end{align*}
  Hence the problem is minimizing
  \begin{align*}
    \max\{ \cos\theta_0, \cos\theta_1, \cos\theta_2\}
  \end{align*}
  subject to
  \begin{align*}
    &\theta_0, \theta_1, \theta_2 > 0 \mbox{ (the triangle is acute) and}\\
    &\theta_0 + \theta_1 + \theta_2 = \pi/2
      \mbox{ (sum of the internal angles)}.
  \end{align*}
  Since $\cos \theta$ is monotonely decreasing 
  on the interval $[0, \pi/2]$,
  a minimum attains at
  $\theta_0 = \theta_1 = \theta_2 = \pi/6$ and
  the minimum is $\cos(\pi/6) = \sqrt{3}/2$.
\end{proof}

\end{document}